\documentclass{article}
\usepackage{amssymb}

\newtheorem{theorem}{Theorem}

\newtheorem{conclusion}[theorem]{Conclusion}
\newtheorem{condition}[theorem]{Condition}

\newtheorem{corollary}[theorem]{Corollary}

\newtheorem{definition}[theorem]{Definition}
\newtheorem{example}[theorem]{Example}

\newtheorem{lemma}[theorem]{Lemma}

\newtheorem{problem}[theorem]{Problem}
\newtheorem{proposition}[theorem]{Proposition}
\newtheorem{remark}[theorem]{Remark}

\newenvironment{proof}[1][Proof]{\noindent\textbf{#1.} }{\ \rule{0.5em}{0.5em}}
\input{tcilatex}
\begin{document}

\title{From\ Subgroups of Direct Sums to Virtuality}
\author{by Stephanos Gekas \\
Aristotle University of Thessaloniki, School of Mathematics}
\date{July 3, 2017}
\maketitle

\begin{abstract}
We start by an original investigation on subgroups of (even infinite) direct
sums in the first 4 sections, that largely generalizes Remak's known
theorem; inspired by that general picture we have elsewhere extended this
elementary "virtual" diagrammatic situation (in diagrammatic length 2
meaning set-theoretic fixation of vertices) by generalizing to the notion of
"virtuality" in module extensions and diagrams in modular representation
theory.

Our first approach starts with an appropriately defined equivalence
relation, which is precisely what allows for treating the confusing case of
multiple factors, thus giving a deeper insight into the structure of such
subgroups.

Several applications and new techniques arising from that approach are
examined, even ones concerning basic properties of homomorphisms, extending
well-known elementary ones.
\end{abstract}

\section{\textbf{Introduction} \ }

\ \ \ \textit{\ }

Although the first part of this article concerns some basic group theory,
that is justified not only in view of the numerous applications that follow
but also by the fact that there are, astonishingly, still many obscurities
about the subgroups of a direct product of $n$ groups, for $n>2$, already in
its general outset. There is, on the other hand, an increasing tendency to
look at subdirect products in more specific contexts and instances, as one
may for example see in \cite{BR}.

I got my first motivation to consider this kind of questions while working
in modular representation theory and trying to understand the subtle inner
structure of modules in a "virtual" framework, by which I was then lead to
analogue but in some sense more general group-theoretic considerations. The
analogue lines can only be drawn by depicting results through "fan-like"
diagrams, which is also a fundamental kind of problem one encounters in the
effort to attach diagrams to modules in an optimal way, so that that they
are somehow analogue to those diagrammatic depictions of subdirect products
regarding that particular feature of the latter, that their vertices are
also set-theoretically fixed. This last feature being impossible for
diagrammatic length (=Loewy length, speaking now of modules) greater than 2,
we may achieve the best possible analogue in the frame of the "Virtual
Category" (see \cite{StG1}\ and \cite{StG}). This bridge to more general
diagrammatic methods (in Representation Theory) is stressed here with
remarks \ref{diagR} and \ref{gdiaG} and finally in the last small section 6.

In particular, by combining the approach and our results here with
"subdirect presentations through homomorphisms" we are lead by an original
way to both known and unknown facts about homomorphisms in general (section
5). It should also be pointed out that our results are generalizable to
operator groups (/operator subgroups) by properly extending/specializing the
proofs.

We are giving an outline of our approach:

The introduction of an equivalence relation in a subgroup $U\leq
A=A_{1}\times A_{2}\times ...\times A_{n}$ is critical for our insight into
its structure, although we can \textit{a posteriori} also determine a normal
subgroup $I$ of $U$, of which the cosets are actually the classes ("adhesive
fibres") of our relation; this is the key to our approach, leading in
particular to our theorems \ref{ssgTh}\ and \ref{smTh}, which generalize
Remak's theorem about the structure (through some "structural" isomorphisms,
intrinsic to the subgroup inclusion of $U$ in $A$) of subgroups of the
direct product of two groups to the case of any arbitrary (even infinite)
number of factors. Namely, in theorem \ref{ssgTh}\ we show an analogue of
that, for any choice of a subset of the set of direct factors, while in
theorem \ref{smTh}\ we determine a necessary and sufficient condition for $U$
to have an analogue $n$-fold structure, i.e. at all places, as for the case $%
n=2$. Finally, the optimal generalization of theorem \ref{smTh} is done with
theorem \ref{gTh}. In all these cases we proceed by means of the
specifically important structure of that normal subgroup $I$, called \textbf{%
the core} of the particular inclusion of $U$ as a subgroup of $A_{1}\times
A_{2}\times ...\times A_{n}$; we are lead to that subgroup and its relevant
to the subdirect product analysis by the key role of its generating
subgroups $E_{i}$ to our equivalence relation, as elucidated in the proofs
of proposition \ref{pr3} and lemma \ref{eqR}. Very critical for our most
general case, treated in theorem \ref{gTh}, is the notion of cohesive
components of the core and its related ("cohesion") decomposition as their
product (propositions \ref{cohD}, \ref{40}).

It is also important to stress that our results may also be applied to the
case of a direct sum of any countable sequence of groups, see remarks \ref%
{dSum}, \ref{dSum2} and \ref{dSum3}.

There are two reasons for us to begin with the case $n=2$, although at least
the final results are well known in this case: (a) It has been precisely
this method, that has lead our intuition to the generalization for any n,
(b) There is also another, methodological reason: Our proof on the one hand
does not depend on the fundamental theorem for group homomorphisms, on the
other this proof, along with the overall point of view of our method, leads
us to the aforementioned results on homomorphisms, which also represent
hitherto unknown generalizations of the fundamental homomorphism theorem, a
theorem which is then obtained as a very special case (as corollary \ref%
{fhth}).

In the following subsection 4.2 we investigate the action of a kind of
"projectivization" of $\left( AutR\right) ^{m}$ on subdirect products over $%
R $, which, apart from defining some orbits in a family of subdirect
products, also gives us a whole orbit once we have one of them. In
subsection 4.3 we investigate conditions for a group $G$\ to be expressible
as a subdirect product of groups belonging to any\ class $\mathfrak{E}$ of
groups, while we also specialize in some particular classes of special
interest.

Another approach proceeds in section 5, by presenting a subdirect product
through "diagonalizing" homomorphisms emanating from a single group; by
that\ we gain an independent, quite different view of them - but then also
some basic general results about homomorphisms, by applying the conclusions
of the previous sections on such presentations of subdirect products,
results considerably extending classical/elementary ones (see proposition %
\ref{2epimTh} and its corollaries). The possibilities of the techniques
arising from this approach are not exhausted here, they are just opened.

The achievements and insights in this work, with the intense focus on the
level of subsets and their elements, somehow pave the way toward a virtual
approach to module extension and diagram theory, the fundaments of which are
laid in \cite{StG1} and in \cite{StG}.

\bigskip \textit{Key words: Presentations of subdirect products,
decomposition as a subdirect product, adhesive fibres, subdirect }$AutR$%
\textit{-classes, Subdirect (in)decomposability, Subdirect presentations.}

\bigskip \textit{About notation:}

We shall automatically consider elements of subgroups of any subproduct of
the original one as elements of the latter too - and vice versa, according
to the context. The phrygic hat $\widehat{}$ above an index designates
omitting, as usual. Whenever we have some "basic" set $\Omega $\ and a
subset $M\subset \Omega $, $\widehat{M}$\ shall denote its complement in $%
\Omega $, so as to have a partition $\Omega =M\sqcup \widehat{M}$. $1$ may
both denote the identity element or the trivial (sub)group. $\pi $ with the
proper indices shall designate projections from direct products. The symbol "%
$\times $" may denote not only outer, but also inner direct product, which
in all cases should be clear from the context. Direct products are denoted
either as $Dr\tprod\limits_{i=1}^{m}W_{i}$\ (like in \cite{JR}, f.ex.) or
just $\tprod\limits_{i=1}^{m}W_{i}$, while in case of infinite factors we
are only using direct sums here.\ 

\bigskip

\section{\textbf{The case of two direct factors}}

The first theorem here is a well known one; the reason to repeat it here is,
as already mentioned, that we are getting to it through a totally different,
original approach, that is then also applicable to the general case of more
than two factors, yielding results and insights which, to my knowledge, are
new. Take notice of the fact that no use of the fundamental theorem for
group homomorphisms is made in its proof; it is thus meaningful to get
another proof of the latter based on it, in fact as a special case of
something much more general (section 5).

\begin{theorem}
(\cite{RR2}; can be found for example in \cite{JT}, also in \cite{JR} 8.19,
p. 183 - or in \cite{BH},\cite{DR},\cite{WS})\label{Th1} Given the direct
product $A\times B$ of two groups and a subgroup $U\leq A\times B$ , \ there
exists a unique isomorphism $\sigma :\pi _{A}\left( U\right) \diagup U\cap
A\rightarrow \pi _{B}\left( U\right) \diagup U\cap B$ , which is thus
\textquotedblleft structural\textquotedblleft , as it determines the
discrete "pair cosets" of which $U$ consists; it is also natural, in the
sense that, for a given subdirect product $U$ of $A\times B$ and a
homomorphism $(f_{A}$, $f_{B})$ from it to $A%
{\acute{}}%
\times B%
{\acute{}}%
$, sending $U$ to $U%
{\acute{}}%
$ and inducing $\sigma ^{\prime }$ from $\sigma $, this $\sigma ^{\prime }$
is precisely the \textquotedblleft structural isomorphism\textquotedblleft\
of $U%
{\acute{}}%
$, as above. Furthermore, the isomorphism $\sigma $ can be (naturally)
continued to $R:=U\diagup \left( U\cap A\right) \times \left( U\cap B\right) 
$. \ \ Conversely, given a normal subgroup $K$ of a subgroup $H_{A}$ of $A$,
respectively, a normal subgroup $L$ of a subgroup $H_{B}$ of $B$ (i.e. $%
K\trianglelefteq H_{A}\leq A,L\trianglelefteq H_{B}\leq B$) and an
isomorphism $\sigma :H_{A}\diagup K\rightarrow H_{B}\diagup L$ , a subgroup $%
U\leq A\times B$ \ is uniquely determined as consisting of the $\sigma $%
-determined pair-cosets. This amounts, of course, to realizing $U$ as a
fiber product of $\pi _{A}\left( U\right) $ and\ $\pi _{B}\left( U\right) $,
over a fixed isomorphic copy $T$ of the two sides of $\sigma ,$ with respect
to $\pi _{X}\left( U\right) $ 's ($X=A,B$) epimorphisms on it, that must be
so coordinated, as to induce precisely that isomorphism $\sigma $ (as
above), that determines $U$ 's pair-fibres correctly.\ \ \ \ \ 
\end{theorem}

\begin{proof}
We define a relation "$\thicksim $" on $U\leq \pi _{A}\left( U\right) \times
\pi _{B}\left( U\right) $ \ by stipulating first that "adjacent" pairs are
related, i.e. $\left( a,b\right) \sim \left( a^{\prime },b\right) \bigskip $%
, $\left( a,b\right) \sim \left( a,b^{\prime }\right) $ for any $\left(
a,b\right) ,\left( a^{\prime },b\right) ,\left( a,b^{\prime }\right) \in U$
and then taking as "$\thicksim $" the transitive hull of this first
stipulation. Reflectivity and symmetricity being apparent, it is clear that "%
$\thicksim $" \ is an equivalence relation\ on $U$\ . \ A crucial property
of this relation is that,\hspace{0in} \hspace{0in}\textbf{(m)}\ if \ $\left(
a_{1},b_{1}\right) \sim \left( a_{1}^{\prime },b_{1}^{\prime }\right) $\ \ \
and\ $\left( a_{2},b_{2}\right) \sim \left( a_{2}^{\prime },b_{2}^{\prime
}\right) $, then \ $\left( a_{1}a_{2},b_{1}b_{2}\right) \sim \left(
a_{1}^{\prime }a_{2}^{\prime },b_{1}^{\prime }b_{2}^{\prime }\right) $\ ;
the relationship between two pairs means the existence of a finite sequence
of pairs starting with the first and ending with the second of those two,
such that any two subsequent pairs have the same first or second coordinate
(we will call such a sequence an adjacency sequence). To see our claim,
consider such sequences for the two given relationships and make them of
equal length by repeating the last term of the shortest one as many times as
necessary\textbf{; }we shall\textbf{\ }have to produce\textbf{\ }a sequence
with these two as terminal members, consisting of subsequent "adjacent"
terms, i.e. ones sharing the same coordinate. To achieve this, take the
products of the terms of same order in those two sequences of equal length,
to get first a sequence of the same length, whose first and last terms are,
respectively, \ $\left( a_{1}a_{2},b_{1}b_{2}\right) $ and $\left(
a_{1}^{\prime }a_{2}^{\prime },b_{1}^{\prime }b_{2}^{\prime }\right) $.
Those subsequent terms in this new sequence that come from the
multiplication of subsequent terms in the original sequences which in both
share\ the same-order coordinate, do probably also share same order terms as
well; there is a slight problem whenever they come from multiplication of
subsequent terms sharing the first term in the one, the second in the other,
hence like this:

\ $\left( \alpha ,\beta \right) $,$\left( \alpha ,\delta \right) $ in the
one (sharing the first coordinate) and $\left( a,b\right) $,$\left(
c,b\right) $ in the other (sharing the second coordinate), thus yielding the
subsequent terms \ $\left( \alpha a,\beta b\right) $,$\left( \alpha c,\delta
b\right) $ \ of the new sequence (of products), which do not share any term;
but then it will suffice to insert the new term $\left( \alpha c,\beta
b\right) =\left( \alpha ,\beta \right) \left( c,b\right) \in U$ between them.

It is, on the other hand, immediate to see that $\left( a,b\right) \sim
\left( a_{1},b_{1}\right) $ implies $\left( a^{-1},b^{-1}\right) \sim \left(
a_{1}^{-1},b_{1}^{-1}\right) $, just by taking the inverses of all terms in
the finite sequence.

We need to notice the obvious fact that $\left( \alpha ,\beta \right) \in U$
and $\left( \alpha ,1\right) \in U$\ imply $\left( 1,\beta \right) =\left(
\alpha ,\beta \right) \left( \alpha ,1\right) ^{-1}\in U$: by symmetry this
yields that, for $\left( \alpha ,\beta \right) \in U$, $\left( \alpha
,1\right) \in U\Longleftrightarrow \left( 1,\beta \right) \in U$ \ \textbf{%
(s)}.

We shall now show how these properties also imply that, whenever $(a,1)$
(respectively, $(1,b)$) is an element of $U$, the first coordinate in its
equivalence class $[(a,1)]$ \ (resp., the second in $[(1,b)]$ ) runs over a
normal subgroup of $\pi _{A}\left( U\right) $ ; in order to see this, we
shall show that, whenever \ $\left( \alpha ,\beta \right) \sim \left( \gamma
,1\right) $, \ it turns out that $\left( \alpha ,1\right) $ (and indeed $%
\left( 1,\beta \right) $ as well) is also an element of $U$ (and, of course,
in the same equivalence class) \textbf{(S)}.

Indeed, let $x_{0}=\left( \gamma ,1\right) ,x_{1},...,x_{n+1}=\left( \alpha
,\beta \right) $ be\ an adjacency sequence for the relationship $\left(
\alpha ,\beta \right) \sim \left( \gamma ,1\right) $; it is then immediate
to see that $x_{0}^{-1}x_{1}x_{2}^{-1}...x_{n+1}^{(-1)^{n}}=\left( \alpha
,1\right) $ or $\left( 1,\beta \right) $\ or $(\alpha ,1)^{-1}=\left( \alpha
^{-1},1\right) $ \ or $\left( 1,\beta \right) ^{-1}$, which then according
to our observation (s) above all of these 4 elements shall belong to $U$, in
particular $\left( \alpha ,1\right) \in U$, as stated.

Hence given any $\left( \alpha ,\beta \right) \in U$, which belongs to the
equivalence class of $\left( 1,1\right) $\textit{,} this can be inside that
class decomposed as $\left( \alpha ,\beta \right) =$ $\left( \alpha
,1\right) \left( 1,\beta \right) $.\ \textit{In particular, this implies
that any element of the equivalence class of }$\left( 1,1\right) \in U$%
\textit{\ is generated by elements adjacent to }$\left( 1,1\right) $\textit{%
, i.e. having }$1$\textit{\ in the one coordinate. If we now also observe
that being adjacent for any two elements of }$U$\textit{\ is the same as
getting any of them by multiplying the other by such a generating element
(i.e. adjacent to }$\left( 1,1\right) $\textit{), this yields that the
equivalence class of }$\left( 1,1\right) $\textit{\ is in fact a subgroup }$%
I $\textit{\ of }$U$\textit{, which we shall call its \textbf{"subdirect
core"}. It is thus also clear in our case (i.e., n=2) that this subdirect
core equals }$\left( U\cap A\right) \times \left( U\cap B\right) $. (Remark:
This decomposition of the subdirect core does not hold in general for $n>2$).

The observation (S) above means also that the first coordinate of the class $%
[(a,1)]=I$ runs over the same set as its elements with second coordinate $1$%
; on the other hand, every element of this form in $U$ apparently belonging
to the same class, this special equivalence class is \ a subgroup of $U$,
which is generated by the elements of the form $\left( \alpha ,1\right) $ or 
$\left( 1,\beta \right) $, hence it must also be normal both in $U$ and in $%
\pi _{A}\left( U\right) $, as conjugation by elements of $U$ yields elements
of the same form, hence in the same class-subgroup and in $A$ as well, apart
from being in advance clear that the set of elements of this form in $U$ is
the group $U\cap $\ $\pi _{A}\left( U\right) $. (This gives also an
alternative way to see the normality of $U\cap A=U\cap \pi _{A}\left(
U\right) $ in $U$ and in $\pi _{A}\left( U\right) $, otherwise clear from
the fact that $\pi _{A}\left( U\right) \times \pi _{B}\left( U\right) $
normalizes $\pi _{A}\left( U\right) $, hence $U\leq \pi _{A}\left( U\right)
\times \pi _{B}\left( U\right) $ must normalize \ $U\cap \pi _{A}\left(
U\right) $ - and then it follows immediately that also $\pi _{A}\left(
U\right) $ has to normalize\ $U\cap \pi _{A}\left( U\right) $; on the other
hand, one sees that the kernel of the restriction to $U$ of $\pi _{A}\left(
U\right) \times \pi _{B}\left( U\right) $ \ 
\'{}%
s projection onto its second factor is precisely $U\cap \pi _{A}\left(
U\right) =U\cap A$.) When we transfer the same remarks to the second factor,
we see immediately that the equivalence class of the elements of $U$ of the
form $(a,1)$, also containing $(1,1)$ and the elements $(1,b)$ in $U$,
equals the group $\left( U\cap A\right) \times \left( U\cap B\right) $.

We wish next to see another interpretation of our equivalence relation on $U$
through its special properties which we have seen:

Assume, so, that $\left( a,b\right) \sim \left( c,d\right) $; invoking the
multiplication property, we may multiply this with the trivial relationship $%
\left( a^{-1},b^{-1}\right) \sim \left( a^{-1},b^{-1}\right) $ to get \ $%
\left( a^{-1}c,b^{-1}d\right) \sim \left( 1,1\right) \Leftrightarrow $\ \ $%
\left( a^{-1}c,b^{-1}d\right) \in \lbrack \left( 1,1\right) ]$=\ \ $\left(
U\cap A\right) \times \left( U\cap B\right) \Longleftrightarrow c\in a\left(
U\cap A\right) $\&$d\in b\left( U\cap B\right) \Leftrightarrow \left(
c,d\right) \in a\left( U\cap A\right) \times b\left( U\cap B\right) $, which
shows that the class of $\left( a,b\right) \in U$ is the Cartesian product
of the left (and right) cosets $a\left( U\cap A\right) $ and $b\left( U\cap
B\right) $ of $U\cap A$, resp. of $U\cap B$, in $\pi _{A}\left( U\right) $,
resp. in $\pi _{B}\left( U\right) $. Now, $a\left( U\cap A\right) \times
b\left( U\cap B\right) $ being an equivalence class, in case there were also
a class $a\left( U\cap A\right) \times d\left( U\cap B\right) $ , while from
the definition of our relation is $\left( a,b\right) \sim \left( a,d\right)
, $ this "new" class has to be identical with $a\left( U\cap A\right) \times
b\left( U\cap B\right) $, i.e. $d\left( U\cap B\right) =b\left( U\cap
B\right) $. This crucial remark establishes a bijection $\sigma :\pi
_{A}\left( U\right) \diagup U\cap A\rightarrow \pi _{B}\left( U\right)
\diagup U\cap B$ , which is bound to be a group homomorphism (hence an
isomorphism), considering the special property (1) of our relation. This
means that $U$ is partitioned into classes which may be described as \textit{%
"pair fibres" of the form }$a\left( U\cap A\right) \times \left( a\left(
U\cap A\right) \right) ^{\sigma }$\textit{,} which at the same time
determines a unique coset $\ \left( a,b\right) \left( U\cap A\right) \times
\left( U\cap B\right) $, where $b$ may be any element of the coset $\left(
a\left( U\cap A\right) \right) ^{\sigma }$ in $\pi _{B}\left( U\right) $.
This observation "prolongs" the bijection $\sigma :\pi _{A}\left( U\right)
\diagup U\cap A\longleftrightarrow \pi _{B}\left( U\right) \diagup U\cap B$
\ to $\pi _{A}\left( U\right) \diagup U\cap A\longleftrightarrow \pi
_{B}\left( U\right) \diagup U\cap B\longleftrightarrow U\diagup \left( U\cap
A\right) \times \left( U\cap B\right) $, still in a homomorphic manner, as
property \textbf{(m)} suggests.

As for the converse assertion, we can easily prove that the given
isomorphism $\sigma $ defines a unique subgroup \ $U\leq A\times B$ - first
as a subset, while the group structure follows from that of the direct
product, where it is embedded; then $U$, according to the previous,
determines a unique such isomorphism, which hence must be $\sigma $. \ \ 

Concerning the interpretation as fiber products, that is quite clear - see
for example \cite{DHP}.
\end{proof}

Another way to realize $U$ could be to view it as a total space of a bundle $%
(U,p,R,\left( U\cap A\right) \times \left( U\cap B\right) )$, $R$\ being the
base group, with "typical fibre" $\left( U\cap A\right) \times \left( U\cap
B\right) $,\ where $p:U\twoheadrightarrow R=U\diagup \left( U\cap A\right)
\times \left( U\cap B\right) $ is the canonical projection, and for $t\in R$
the fibre (:"$\thicksim $"-equivalence class or "adhesive fibre") is $%
A_{t}\times B_{t}$\ with $A_{t}$, $B_{t}$ its ($\sigma $-) corresponding $%
\left( U\cap A\right) $-, resp. $\left( U\cap B\right) $-, cosets in $\pi
_{A}\left( U\right) $, resp. $\pi _{B}\left( U\right) $, so that $%
p^{-1}(t)=A_{t}\times B_{t}$. This point of view can also be adapted to our
theorems \ref{ssgTh}, \ref{smTh} and \ref{gTh} below.

For later use in section 6 we finally prove (independently of theorem 1) the
following

\begin{lemma}
\label{L2}$\pi _{A}\left( U\right) \cong U\diagup U\cap B$,\ $\pi _{B}\left(
U\right) \cong U\diagup U\cap A$.\ \ \ \ \ \ \ \ \ \ \ \ \ \ \ \ \ \ \ \ \ \
\ \ \ \ \ \ \ \ \ \ \ \ \ \ \ \ \ 
\end{lemma}

\begin{proof}
\bigskip Observe that $\ker \left( \pi _{A}|_{U}\right) =U\cap B$,\ $\ker
\left( \pi _{B}|_{U}\right) =U\cap A$.\ \ \ \ \ \ \ \ \ \ \ \ \ \ \ \ \ \ \
\ \ \ \ \ \ 
\end{proof}

\bigskip

\section{The generic case}

\bigskip Let $U\leq A=A_{1}\times A_{2}\times ...\times A_{n}$ ; we assume
further that\textbf{\ }$\pi _{i}\left( U\right) \cap A_{i}$\textbf{\ is not
trivial for any }$i\in \left\{ 1,...,n\right\} $\ and then introduce the
following subgroups of this product: \ 

$E_{s}=\pi _{1...\widehat{s}...n}\left( U\right) \cap U$, which is clearly
equal to $Ker\pi _{s}\mid _{U}=$ $Ker\pi _{s}\cap U=\left(
Dr\tprod\limits_{i\neq s}A_{i}\right) \cap U$, \ \ consisting of the
elements of U, having 1 in the $s$-coordinate. More generally, for any\ set $%
\Lambda $ of indices $i_{1}\langle i_{2}\langle ...\langle i_{s}$ from $%
\left\{ 1,...,n\right\} $, define $E_{i_{1}i_{2}...i_{s}}:=\pi _{1...%
\widehat{i_{1}}...\widehat{i_{2}}...\widehat{i_{s}}...n}\left( U\right) \cap
U$, consisting of the elements of U, having 1 in the $i_{1},i_{2},...,i_{s}$%
-coordinates. It is obvious that $E_{1...\widehat{i_{1}}...\widehat{i_{2}}...%
\widehat{i_{s}}...n}=\tbigcap\limits_{t\notin \left\{
i_{1},...,i_{s}\right\} }E_{t}=Ker\pi _{\widehat{\Lambda }}|_{U}$ \ while $%
E_{i_{1}...i_{2}...i_{s}}=\tbigcap\limits_{t\in \left\{
i_{1},...,i_{s}\right\} }E_{t}=E_{i_{1}}\cap E_{i_{2}}\cap ...\cap
E_{i_{s}}=Ker\pi _{_{\Lambda }}|_{U}$, $E_{\left\{ 1,...,n\right\} }$ is the
trivial subgroup of $I$.

\begin{definition}
We define, as in the case n=2, a relation "$\sim $" on U, in two steps:
first, we stipulate that $\overline{a}=\left( a_{1},...,a_{n}\right) \sim 
\overline{b}=\left( b_{1},...,b_{n}\right) $ whenever $a_{\lambda
}=b_{\lambda }$ for some $\lambda \in \left\{ 1,...,n\right\} $ (we will
then say that $\overline{a}$\ and $\overline{b}$ \textbf{touch one another
or are adjacent at the }$\lambda $\textbf{-coordinate}), then take the
minimal transitive extension of this first germ relation, to get an
equivalence relation on $U$;\ we shall designate equivalence classes of "$%
\sim $" by using square brackets $\left[ -\right] $, which shall also be
called "adhesive fibres" of the subgroup $U$. In particular, the equivalence
class $I=\left[ \left( 1,...,1\right) \right] $ of $U$ 's identity element
shall be referred to as the \textbf{the (subdirect) core} of the subgroup $U$
of \ $A_{1}\times A_{2}\times ...\times A_{n}$.\ 
\end{definition}

\begin{remark}
\label{dSum}\bigskip Our results from this section onward are easily seen to
be extendable to the case of a subgroup $U$\ of an infinite direct sum $%
\dbigsqcup\limits_{j\in J}A_{j}$, where the index set $J$ is a totally
ordered and countably infinite one (f.ex. $%
\mathbb{N}
$), in a similar manner (and along the ordering of $J$, upwards). As for the
definition of the above relation in this case, notice that also in the case
of infinitely many direct summands \textbf{only finite "connecting" sequences%
} of adjacent elements shall be entailed for any assertion of equivalence
between two elements of $U$. \ \ \ \ \ 
\end{remark}

Next we shall show $I$ to be a normal subgroup\ of $U$ and the equivalence
classes of "$\sim $" indeed the same as $I$ 's cosets in $U$.

The following subgroups of $U$, which are readily seen to be \textit{normal}
subgroups in $U$, shall play a crucial role in our investigation:

Define for any subsequence $\Lambda $ of indices $i_{1}\langle i_{2}\langle
...\langle i_{s}$ from $J=\left\{ 1,...,n\right\} $, \ \ $L_{\Lambda
}=L_{i_{1}...i_{2}...i_{s}}$ to be $\left( A_{i_{1}}\times A_{i_{2}}\times
...\times A_{i_{s}}\right) \cap I$, set also $L_{\varnothing }$\ to be the
trivial subgroup of $U$. Let us call them \textit{"subcores"} of $U$. Notice
that $E_{1...\widehat{s}...n}=L_{s}$, as they both consist of the elements
of $U$, having all but their $s$-coordinate equal to 1; also, $L_{1...%
\widehat{s}...n}=E_{s}$. More generally, $L_{\Lambda
}=L_{i_{1}i_{2}...i_{s}}=E_{1...\widehat{i_{1}}...\widehat{i_{2}}...\widehat{%
i_{s}}...n}=Ker\pi _{\widehat{\Lambda }}$\ , hence a normal subgroup of $U$\
for any proper subset $\Lambda $\ of $J=\left\{ 1,...,n\right\} $. However
this cannot be likewise concluded when $\Lambda =\left\{ 1,...,n\right\} $,
then yielding $I$ as $L_{\Lambda }$; \ instead, we are proving that in the
following proposition.

Notice that, for $M,N\subset \left\{ 1,...,n\right\} $, with\ $M\cap
N=\varnothing $, it follows that $L_{M}\cap L_{N}=1$, therefore\ $%
L_{M}L_{N}=L_{N}L_{M}$: in fact, they even commute elementwise (from the
original direct product).\ Set also $L_{\emptyset }=1$,\ $E_{\emptyset }=I$.

For any $M\subset \left\{ 1,...,n\right\} $, $\pi _{M}$ denotes the
corresponding projection from $A_{1}\times A_{2}\times ...\times A_{n}$ to $%
Dr\tprod\limits_{i\in M}A_{i}$.

We want to recall here the infinite symmetric group $\Sigma _{\infty }$,
definable f.ex. on the set of positive integers and consisting of the
permutations of it, that fix all but a finite subset of it. We shall also
use the fact that the symmetric group $S_{n}$\ is generated by its
involutions, i.e. transpositions, which is also the case for $\Sigma
_{\infty }$. The latter equals the injective limit of all the symmetric
groups $S_{n}$, $n\in 
\mathbb{N}
^{\ast }$. \textbf{Actually we may just choose the }$n-1$\textbf{\
transpositions }$\left( i,i+1\right) $\textbf{, }$i=1,...n-1$\textbf{,\ as
generators for }$S_{n}$\textbf{\bigskip\ - hence for }$\Sigma _{\infty }$%
\textbf{\ too.\ }

\begin{proposition}
\label{pr3}The equivalence class $I=\left[ \left( 1,...,1\right) \right] $
of $U$ 's identity element is a normal subgroup of $U$, henceforth to be
called \textbf{the subdirect core} of the subgroup $U$ of \ $A_{1}\times
A_{2}\times ...\times A_{n}$, generated by its subgroups $E_{i}$, $i=1,...,n$%
; any $E_{i}$ is normal both in $U$ and in $\pi _{1...\widehat{i}...n}\left(
U\right) $. Furthermore, $I=\tprod\limits_{i=1}^{n}E_{i}=\tprod%
\limits_{i=1}^{n}E_{\tau \left( i\right) }$, where $\tau $ is any
permutation in $S_{n}$, a result that also holds for the subgroup $I_{M}$ of 
$I$ generated by any non-empty subset of $\left\{ E_{1},...,E_{n}\right\} $,
corresponding to a subset $M$\ of $\left\{ 1,...,n\right\} $. In the case
that we have a subgroup $U$\ of an infinite direct sum $\dbigsqcup\limits_{j%
\in J}A_{j}$, where the index set $J$ is countably infinite, we define
similarly the core $I$\ as the subgroup $I=$ $\dbigsqcup\limits_{j\in
J}E_{j} $ of $U$ , wherein it is now again possible to permute summands by
any element of the infinite symmetric group $\Sigma _{\infty }$.
\end{proposition}

\begin{proof}
$E_{s}\trianglelefteq U$, $s=1,2,...,n$ , because $E_{s}$ consists precisely
of the elements of $U$, that have $1$ in the $s$-coordinate - a property
maintained through conjugation in U. Alternatively, we might just use that $%
E_{\Lambda }=Ker\pi _{\Lambda }|_{U}$.

Observe that $\overline{a}=\left( a_{1},...,a_{n}\right) \in
I\Leftrightarrow $ $\overline{a}\sim \left( 1,...,1\right) \iff \exists $ a
finite sequence $\overline{a^{0}}=\overline{a},\overline{a^{1}},...,%
\overline{a^{\mu +1}}=\left( 1,...,1\right) $, set $\overline{a^{\kappa }}%
=\left( a_{1}^{\kappa },...,a_{n}^{\kappa }\right) $ for $\kappa
=0,1,...,\mu +1$, such that any two neighbouring terms $\overline{a^{\kappa
-1}},\overline{a^{\kappa }}$ share, say, their $i_{\kappa }$-coordinate. (By
assuming this sequence to be of minimal length, we get that $a_{i}^{\kappa
}\neq 1\forall i\in \left\{ 1,...,n\right\} $ whenever $\kappa \langle \mu $%
). Then, by taking the sequence of the inverses we get such an "adjacency
sequence" yielding the relationship $\overline{a}^{-1}\sim \left(
1,...,1\right) $, proving that $\overline{a}^{-1}\in I$ as well.

Before proceeding to prove that, given another $\overline{b}=\left(
b_{1},...,b_{n}\right) \in I$, $\overline{a}\overline{b}$ shall belong to $I$
too, we must make a crucial remark: in the above "adjacency sequence" for $%
\overline{a}\sim \left( 1,...,1\right) $, $\overline{a^{\kappa }}^{-1}%
\overline{a^{\kappa -1}}\in U$ with $1$ in the $i_{\kappa }$-coordinate,
hence $\overline{a^{\kappa }}^{-1}\overline{a^{\kappa -1}}\in E_{i_{\kappa
}} $, allowing us to replace the condition for the existence of an
"adjacency sequence" for $\overline{a}\sim \left( 1,...,1\right) $ with the
possibility to write $\overline{a}$ as a product of elements belonging to
the several $E_{i}$'s: indeed,

$\overline{a}=\overline{a^{0}}=\overline{a^{\mu }}\left( \overline{a^{\mu }}%
^{-1}\overline{a^{\mu -1}}\right) \left( \overline{a^{\mu -1}}^{-1}\overline{%
a^{\mu -2}}\right) ...\left( \overline{a^{2}}^{-1}\overline{a^{1}}\right)
\left( \overline{a^{1}}^{-1}\overline{a^{0}}\right) =$ \ \ 

\ $=\overline{\alpha ^{\mu +1}}\overline{\alpha ^{\mu }}\overline{\alpha
^{\mu -1}}...\overline{\alpha ^{2}}\overline{\alpha ^{1}}$, with any $%
\overline{\alpha ^{\kappa }}\in E_{i_{\kappa }}$,\ where it is obvious what
we have substituted the greek $\overline{\alpha }$ 's for; conversely, given
such an expression of $\overline{a}=\overline{a^{0}}$ as $\overline{\alpha
^{\mu +1}}\overline{\alpha ^{\mu }}...\overline{\alpha ^{2}}\overline{\alpha
^{1}}$, with any $\overline{\alpha ^{\kappa }}\in E_{i_{\kappa }}$, we get
the adjacency sequence $\overline{a}=\overline{a^{0}}=$ $\overline{\alpha
^{\mu +1}}\overline{\alpha ^{\mu }}...\overline{\alpha ^{2}}\overline{\alpha
^{1}}$, $\overline{a^{1}}=\overline{\alpha ^{\mu +1}}\overline{\alpha ^{\mu }%
}...\overline{\alpha ^{3}}\overline{\alpha ^{2}}$, $\overline{a^{2}}=$ $%
\overline{\alpha ^{\mu +1}}\overline{\alpha ^{\mu }}...\overline{\alpha ^{3}}
$, ..., $\overline{a^{\mu -1}}=$ $\overline{\alpha ^{\mu +1}}$ $\overline{%
a^{\mu }}$ , $\overline{a^{\mu }}=$ $\overline{\alpha ^{\mu +1}}$, $%
\overline{a^{\mu +1}}=\left( 1,...,1\right) $. Hence we may also write $%
\overline{b}\in I$ as a product of elements of the several $E_{i}$'s, say $%
\overline{b}$=$\overline{\beta ^{\nu }}\overline{\beta ^{\nu -1}}...%
\overline{\beta ^{2}}\overline{\beta ^{1}}$, therefore it becomes obvious
through this new equivalent condition for an element of $U$ to belong to $I$
that also the product $\overline{a}\overline{\beta }=\overline{\alpha ^{\mu }%
}\overline{\alpha ^{\mu -1}}...\overline{\alpha ^{2}}\overline{\alpha ^{1}}%
\overline{\beta ^{\nu }}\overline{\beta ^{\nu -1}}...\overline{\beta ^{2}}%
\overline{\beta ^{1}}$ belongs to $I$, proving that $U$ is generated by its
subgroups $E_{i},i=1,...,n$ \ - \ i.e., $I=\left\langle
\tbigcup\limits_{i=1}^{n}E_{i}\right\rangle $. This, combined with the
normality of the $E_{i}$'s in $U$, assures that $I$ is normal in $U$.

As for the last claim, it will suffice to prove that, for any $a\in I$, it
is possible to write it as a product $\overline{\alpha ^{\mu +1}}\overline{%
\alpha ^{\mu }}...\overline{\alpha ^{2}}\overline{\alpha ^{1}}$, where $%
\overline{\alpha ^{\kappa }}\in E_{i_{\kappa }}$, in such a way, that $%
i_{\mu +1}\langle i_{\mu }\langle ...\langle i_{2}\langle i_{1}$ - or even
in a way such that this ordering will first be valid after application of
the (random) permutation $\tau ^{-1}$; to see this, it will obviously be
enough to prove that, given a product $e_{i}e_{j}$ with $e_{i}\in
E_{i},e_{j}\in E_{j}$, it is always possible to write it in the form $%
e_{j}^{\prime }e_{i}^{\prime }$, where $e_{i}^{\prime }\in
E_{i},e_{j}^{\prime }\in E_{j}$; in particular, we need that just for $j=i+1$%
, as we can then generate any permutation $\tau $. For notational
convenience we prove it for $i=1$, $j=2$; so, let $e_{1}\in E_{1,}e_{2}\in
E_{2}$. By taking their commutator $[e_{1}{}_{,}e_{2}]$, one sees directly
that it belongs to $E_{1}\cap E_{2}=E_{12}$, hence $%
e_{1}{}e_{2}{}=e_{2}e_{1}[e_{1}{}_{,}e_{2}]$ , which is a product of $%
e_{2}\in E_{2}$ and $e_{1}[e_{1}{}_{,}e_{2}]\in E_{1}$. (Alternatively, it
is enough to remember that the $E_{i}$'s, as well as any (finite) products
of them, are all normal subgroups of $U$).

We emphasize here that this does not in general mean that the elements of $%
E_{i}$ commute with those of $E_{j}$ (with $i\neq j$), unless $n=2$, in
which case the commutator above becomes the identity element of $U$, as $%
E_{12}$ is then the trivial subgroup.
\end{proof}

It cannot be overstressed that the core $\mathbf{I}$ does NOT, in general,
pertain to the group $U$, but to its particular given inclusion \ $U\leq
A_{1}\times A_{2}\times ...\times A_{n}$ as a subgroup of \textit{that}
direct product.

\begin{lemma}
\label{eqR}\bigskip For $\overline{a},\overline{b}\in U,$ it holds that $%
\overline{a}\sim \overline{b}$ iff \ $\overline{b}^{-1}\overline{a}\in I$ ;
\ hence equivalence classes of "$\sim $" \ is the same thing as $I$ 's
cosets in $U$.
\end{lemma}

\begin{proof}
$\overline{a}\sim \overline{b}\iff \exists $ a finite sequence $\overline{%
a^{0}}=\overline{a},\overline{a^{1}},...,\overline{a^{\mu +1}}=\overline{b}$
, such that any two neighbouring terms $\overline{a^{\kappa -1}},\overline{%
a^{\kappa }}$ share, say, their $i_{\kappa }$-coordinate for $\kappa
=1,...,\mu +1$, meaning that $\overline{a^{\kappa }}^{-1}\overline{a^{\kappa
-1}}\in E_{i_{\kappa }}$; set $\overline{e^{\kappa }}=\overline{a^{\kappa }}%
^{-1}\overline{a^{\kappa -1}}$. Now, $\overline{a}=\overline{a^{0}}=%
\overline{a^{\mu +1}}\left( \overline{a^{\mu +1}}^{-1}\overline{a^{\mu }}%
\right) \left( \overline{a^{\mu }}^{-1}\overline{a^{\mu -1}}\right)
...\left( \overline{a^{2}}^{-1}\overline{a^{1}}\right) \left( \overline{a^{1}%
}^{-1}\overline{a^{0}}\right) =$ $\overline{b}\overline{e^{\mu +1}}\overline{%
e^{\mu }}...\overline{e^{2}}\overline{e^{1}}\Rightarrow \overline{b}^{-1}%
\overline{a}\in I$ , the converse becoming apparent by expressing $\overline{%
b}^{-1}\overline{a}\in I$ \ as a product of elements of the $E_{i}$'s and
then using the just obtained equivalent condition for $\overline{a}\sim 
\overline{b}$\ .\bigskip \bigskip
\end{proof}

\begin{remark}
Despite the "simplifying" assertion of the preceding lemma it is however
important for our understanding and analysis also to continue identifying
the cosets of $I$\ as "adhesive fibres", according to our original
definition of the equivalence relation. It is meanwhile precisely that
understanding, which has lead us to this new approach and insight into the
subdirect structure - while, besides being the starting point of view for
the bulk of our general analysis here, it also proves crucial later for the
proper understanding of the following subsections.
\end{remark}

\bigskip The first four points of the following lemma are a direct
consequence of the definitions:

\begin{lemma}
\label{basL}Let $\varnothing \neq M,M%
{\acute{}}%
,N\subseteq \left\{ 1,...,n\right\} $. Then, (i) $E_{M}\cap E_{N}=E_{M\cup
N} $\ ,(ii) For $M%
{\acute{}}%
\subseteq M$, $L_{M%
{\acute{}}%
}\subseteq L_{M}$, (iii) Hence $L_{M\cup N}\supseteq L_{M}L_{N}$, $%
L_{N}L_{M} $, which are both equal to $L_{M}\times L_{N}$ in case $M\cap
N=\varnothing $, (iv) $L_{M\cap N}=L_{M}\cap L_{N}$. Furthermore, for any
subset $\Lambda $ of the set $J$ of indices, $E_{\Lambda }\trianglelefteq
\pi _{\widehat{\Lambda }}\left( U\right) $.
\end{lemma}

\begin{proof}
\bigskip We need only to prove the last assertion.

We shall show that $E_{s}\trianglelefteq \pi _{1...\widehat{s}...n}\left(
U\right) $. Just for notational convenience, we will show this for s=1, i.e.
that $E_{1}\trianglelefteq \pi _{2...n}\left( U\right) $. Let, so, $%
\overline{a}=\left( 1,a_{2},...,a_{n}\right) \in E_{1}$ and $\overline{%
b^{\prime }}=\left( b_{2},...,b_{n}\right) =\left( 1,b_{2},...,b_{n}\right)
\in \pi _{2...n}\left( U\right) $, which means that there exists some $%
b_{1}\in A_{1}$, such that $\overline{b}=\left( b_{1},...,b_{n}\right) \in U$%
; then $\overline{a}^{\overline{b^{\prime }}}=\overline{a}^{\overline{b}}\in
U$, therefore $\overline{a}^{\overline{b^{\prime }}}\in \pi _{2...n}\left(
U\right) \cap U=E_{1}$, completing the argument. The argument in the general
case of any subset $\Lambda $ of the set $J$ of indices is completely
similar.
\end{proof}

\bigskip

The next lemma is a very practical one, although quite obvious:

\begin{lemma}
\label{nonD}For $M\cap N=\varnothing $ above, non-equality in $L_{M\cup
N}\supseteq $\ $L_{M}\times L_{N}$\ means the existence of some element $%
\overline{a}=\left( a_{M};a_{N};1,...,1\right) \in I$, such that $\pi
_{M}\left( \overline{a}\right) $($=a_{M}$)$\notin I$ ($\Leftrightarrow \pi
_{N}\left( \overline{a}\right) \notin I$). By an equivalent formulation, for 
$M\cap N=\varnothing $, $L_{M\cup N}=$\ $L_{M}\times L_{N}$\ iff for any $%
\overline{a}\in L_{M\cup N}$, $\pi _{M}\left( \overline{a}\right) \in
L_{M\cup N}$\ ($\Leftrightarrow \pi _{N}\left( \overline{a}\right) \in
L_{M\cup N}$).
\end{lemma}

\begin{definition}
\bigskip \label{SignD}Let $\Lambda $ be any\ set of indices $i_{1}\langle
i_{2}\langle ...\langle i_{s}$ from $\left\{ 1,...,n\right\} $, $L_{\Lambda
}=L_{i_{1}...i_{2}...i_{s}}$ the corresponding "subcore" $\left(
A_{i_{1}}\times A_{i_{2}}\times ...\times A_{i_{s}}\right) \cap I$ of $U$.
We shall call such a non-trivial subcore $L_{\Lambda
}=L_{i_{1}i_{2}...i_{s}} $ \textbf{cohesive in }$U\leq A_{1}\times
A_{2}\times ...\times A_{n}$\textbf{\ } if there is no non-trivial partition 
$\Lambda =\left\{ i_{1},i_{2},...,i_{s}\right\} =M\cup N$ of\ $\Lambda $\ ($%
M\cap N=\varnothing $), with $L_{M}$, $L_{N}$ non-trivial and $L_{\Lambda
}=L_{M}\times L_{N}$, i.e., so that $L_{\Lambda }$ split over $L_{M}\ $($%
\Leftrightarrow $over $L_{N}$).\ 

A subcore $L_{\Lambda }$\ shall be called \textbf{reducible} if there is a
proper subset $M\subset \Lambda $, such that $L_{M}=L_{\Lambda }$; otherwise
we shall designate it as a \textbf{non-reducible} subcore.
\end{definition}

The last lemma \ref{nonD} is crucial up to the proof of the following one:

\begin{lemma}
Assume $\varnothing \neq M$, $N\subseteq \left\{ 1,...,n\right\} $ with $%
M\cap N\neq \varnothing $, such that the subcores $L_{M}$, $L_{N}$\ be
cohesive and non-reducible. Then $L_{M\cup N}$\ is cohesive too. Therefore
there are uniquely determined maximal cohesive subcores $L_{N_{i}}$ 's, $%
i=1,2,...,n$ , which we shall call the \textbf{cohesive components }of I,
and we have correspondingly the finest possible decomposition $%
I=L_{N_{1}}\times L_{N_{2}}\times ...\times L_{N_{n}}$ of the core.
\end{lemma}

\begin{proof}
\bigskip Assume, to the contrary, that there exists a non-trivial partition $%
S\sqcup S^{\prime }$\ of $M\cup N$, so that \textbf{(1)} $L_{M\cup
N}=L_{S}\times L_{S^{\prime }}$. Set now $S\cap M=S_{1}$, $S\cap N=S_{2}$, $%
S^{\prime }\cap M=S_{1}^{\prime }$, $S^{\prime }\cap N=S_{2}^{\prime }$. We
shall show that $L_{M}=L_{S_{1}}\times L_{S_{1}^{\prime }}$\ and $%
L_{N}=L_{S_{2}}\times L_{S_{2}^{\prime }}$; by elementary set-theoretic
arguments on the index-sets it is immediate to see that, in any case, at
least one of the above direct decompositions is non-trivial, which yields a
contradiction to the cohesiveness. The proof being similar in both cases, we
are restricting ourselves to showing the first one.

Take any arbitrary $\overline{a}\in L_{M}\subseteq L_{M\cup N}$, hence by 
\textbf{(1)} and lemma \ref{nonD} $\pi _{S}\left( \overline{a}\right) \in I$%
, implying both, $\pi _{S}\left( \overline{a}\right) \in L_{S}$\ and $\pi
_{S}\left( \overline{a}\right) \in L_{M}$, therefore by lemma \ref{basL}(iv) 
$\pi _{S}\left( \overline{a}\right) \in L_{S\cap M}=L_{S_{1}}$and, a\
fortiori, $\pi _{S_{1}}\left( \overline{a}\right) \in L_{S_{1}}$; similarly, 
$\pi _{S_{1}^{\prime }}\left( \overline{a}\right) \in L_{S_{1}^{\prime }}$.
Then, again by lemma \ref{nonD}, we get $L_{M}=L_{S_{1}}\times
L_{S_{1}^{\prime }}$, as wished.

This property guarantees that any cohesive subcore is contained in a maximal
one; then, clearly by virtue of maximality, the (direct) product of all the
maximal cohesive subcores gives the whole of $I$.\thinspace
\end{proof}

\begin{remark}
\label{dSum2}\bigskip In case we had an infinite sum (see remark \ref{dSum})
instead of the finite case we have considered in our proofs, the first part
of the last lemma would ensure that the set of cohesive subcores is
inductively ordered, hence Zorns lemma applies, to give that any cohesive
subcore is contained in a maximal one also in this case.
\end{remark}

\begin{definition}
In case the cohesive components of $I$ are precisely the $L_{i}$ 's, $%
i=1,2,...,n$ , or, equivalently, $I=L_{1}\times L_{2}\times ...\times L_{n}$%
, we shall call the subgroup $U\leq A_{1}\times A_{2}\times ...\times A_{n}$
a (\textbf{cohesively) smashed} one. If for a non-empty, proper subset $M$\
of $\left\{ 1,...,n\right\} $, $I=L_{M}\times L_{\widehat{M}}$, then we
shall say that $I$\ splits over $L_{M}$ (or over $L_{\widehat{M}}$)\ - or
even, by a simplifying controlled abuse of language, over $M$\ (or $\widehat{%
M}$).

At the other extreme of such a non-trivial decomposition of the core, we
want to make\ another distinction, that of a \textbf{deltoid} subcore $%
L_{\Lambda }$\ of\ $U$, meaning that for any proper subset of indices $M$
of\ $\Lambda $, $L_{M}$ is trivial. A non-interesting special case of that
occurs whenever $\left\vert \Lambda \right\vert =1$, such subcores shall be
referred to as trivial ones. $U$\ itself shall be called deltoid if all $%
E_{i}$ 's are trivial; in that case it may also be viewed as a (non-proper)
deltoid subcore of itself.
\end{definition}

\bigskip The following two lemmata are quite immediate to see:

\begin{lemma}
\label{LsplC}Let a $M$ be a non-empty, proper subset\ of $\left\{
1,...,n\right\} $; then $L_{M}\times L_{\widehat{M}}\subseteq I$, with "="
holding iff $I$\ splits over $M$.
\end{lemma}

\begin{lemma}
\label{DD}\bigskip If, for a non-empty, proper subset $M$\ of $\left\{
1,...,n\right\} $, $L_{M}$\ is a maximal deltoid subcore of $U$ (even
trivially, with $\left\vert M\right\vert =1$), then $L_{M}$\ is a direct
factor of $I$.\bigskip \bigskip
\end{lemma}

\begin{lemma}
\ \label{LssgTh}\ Let $U\leq A_{1}\times A_{2}\times ...\times A_{n}$ and
assume also that all $\pi _{i}\left( U\right) $ 's are non-trivial
(otherwise we should have considered the maximal subproduct satisfying this
condition); then,\ the following are true:

\textbf{a.} For $\overline{a},\overline{b}\in U$, with $\overline{a}$ fixed, 
$\overline{b}$ variable, so that \ \ $\pi _{s}\left( \overline{b}\right) =$\ 
$\pi _{s}$\ $\left( \overline{a}\right) ,$ the variation domain of $\ 
\overline{b}$ $\ $is the coset $\overline{a}E_{s}$ in $U$ while the
variation domain of $\pi _{1...\widehat{s}...n}\left( \overline{b}\right) $ $%
\ $is the coset $\pi _{1...\widehat{s}...n}\left( \overline{a}\right) E_{s}$
in $\pi _{1...\widehat{s}...n}\left( U\right) $; conversely, by varying only
the s-coordinate, the variation space of $\overline{b}$ $\ $is the coset $%
\overline{a}E_{1...\widehat{s}...n}=\overline{a}L_{s}$ in $U$.\ \ More
generally, for any variable $\overline{a}=\left( a_{1},...,a_{n}\right) \in
U $ and any\ sequence of indices $i_{1}\langle i_{2}\langle ...\langle i_{s}$
from $\left\{ 1,...,n\right\} $, \ by fixing the $\left\{
i_{1},i_{2},...,i_{s}\right\} $-coordinates ($\left\{
a_{i_{1}},...,a_{i_{s}}\right\} $) and varying the others, so that \ $%
\overline{a}$\ \ remain in $U$, the variation domain of $\overline{a}$\ is
the coset $\ \overline{a}E_{i_{1}i_{2}...i_{s}}$; $\ $ conversely, by fixing
the complementary set of coordinates, the variation domain of $\overline{a}$
becomes $\overline{a}E_{1...\widehat{i_{1}}...\widehat{i_{2}}...\widehat{%
i_{s}}...n}$. Corresponding to that,\ $\ $the variation domain of $\pi _{1...%
\widehat{i_{1}}...\widehat{i_{2}}...\widehat{i_{s}}...n}\left( \overline{a}%
\right) $ is, in the first case, the coset $\pi _{1...\widehat{i_{1}}...%
\widehat{i_{2}}...\widehat{i_{s}}...n}\left( \overline{a}\right)
E_{i_{1}i_{2}...i_{s}}$ in $\pi _{1...\widehat{i_{1}}...\widehat{i_{2}}...%
\widehat{i_{s}}...n}\left( U\right) $; \ accordingly, by varying the $%
\left\{ i_{1},i_{2},...,i_{s}\right\} $-coordinates $\left\{
a_{i_{1}},...,a_{i_{s}}\right\} $ and fixing the others, so that $\overline{a%
}$ remain in $U$, the variation space of $\pi _{i_{1}...i_{2}...i_{s}}\left( 
\overline{a}\right) $ \ (while lying inside $\pi _{i_{1}i_{2}...i_{s}}\left(
U\right) $)\ is the coset $\pi _{i_{1}i_{2}...i_{s}}\left( \overline{a}%
\right) E_{1...\widehat{i_{1}}...\widehat{i_{2}}...\widehat{i_{s}}...n}$ ,\
for any particular ("original") value of $\overline{a}$.\ \ \ 

\textbf{b.} For any non-empty proper subset $M=\left\{
i_{1},i_{2},...,i_{s}\right\} $\ of $\left\{ 1,...,n\right\} $, \ say $%
i_{1}\langle i_{2}\langle ...\langle i_{s}$, there is a unique isomorphism \ 
$\sigma :$ $\pi _{1...\widehat{i_{1}}...\widehat{i_{2}}...\widehat{i_{s}}%
...n}\left( U\right) \diagup E_{i_{1}i_{2}...i_{s}}\widetilde{\rightarrow }$%
\ \ $\pi _{i_{1}i_{2}...i_{s}}\left( U\right) \diagup E_{1...\widehat{i_{1}}%
...\widehat{i_{2}}...\widehat{i_{s}}...n}$ ($\widetilde{\rightarrow }$\ $%
U\diagup \left( E_{i_{1}i_{2}...i_{s}}\times E_{1...\widehat{i_{1}}...%
\widehat{i_{2}}...\widehat{i_{s}}...n}\right) $),\ with the

"structural" property that, for any $\overline{a}=\left(
a_{1},...,a_{n}\right) \in U,$ \ $\sigma $ sends the coset $\pi _{1...%
\widehat{i_{1}}...\widehat{i_{2}}...\widehat{i_{s}}...n}\left( \overline{a}%
\right) E_{i_{1}i_{2}...i_{s}}$\ to $\pi _{i_{1}i_{2}...i_{s}}\left( 
\overline{a}\right) E_{1...\widehat{i_{1}}...\widehat{i_{2}}...\widehat{i_{s}%
}...n}$. In other words $U$ may be realized as a fiber product of $\pi _{1...%
\widehat{i_{1}}...\widehat{i_{2}}...\widehat{i_{s}}...n}\left( U\right) $
and\ $\pi _{i_{1}i_{2}...i_{s}}\left( U\right) $, over a fixed isomorphic
copy of the two sides of $\sigma $ with respect to their apparent
epimorphisms on it.\ 

As for the converse, a subgroup $U$\ is now determined by the following
data: A normal subgroup $I$\ of $U$, with the property that it only contains
one "$\sim $"-equivalence class, a partition $\left\{ 1,...,n\right\} =M\cup 
\widehat{M}$, subgroups $W_{M}\leq $\ $Dr\tprod\limits_{i\in M}A_{i}$\ and $%
W_{\widehat{M}}\leq $\ $Dr\tprod\limits_{i\in \widehat{M}}A_{i}$,
respectively containing $L_{M}$\ and $L_{\widehat{M}}$,\ together with a
"structural isomorphism" \ $\sigma :$ $W_{M}\diagup L_{M}\widetilde{%
\rightarrow }$\ \ $W_{\widehat{M}}\diagup L_{\widehat{M}}$, where $L_{M}=\pi
_{M}\left( I\right) $, $L_{\widehat{M}}=\pi _{\widehat{M}}\left( I\right) $. 
$U$ thus is\ determined set-theoretically as a subset of the direct product,
which in turn fully determines its group structure.
\end{lemma}

\begin{proof}
\ (We are treating the general case of (a), the first one just being a
special case of that.)

\textit{We shall throughout keep on the convention of viewing the image of
any projection from the original direct product as contained (embedded) in
that product too, in the obvious way.}

\textbf{a.} So, let $\overline{a^{\prime }}\in U$ have the same $\left\{
i_{1},i_{2},...,i_{s}\right\} $-coordinates ($\left\{
a_{i_{1}},...,a_{i_{s}}\right\} $) as $\overline{a}$, i.e. $\pi
_{i_{1}i_{2}...i_{s}}\left( \overline{a^{\prime }}\right) =\pi
_{i_{1}i_{2}...i_{s}}\left( \overline{a}\right) $; then obviously $\overline{%
\varepsilon }=\overline{a}^{-1}\overline{a^{\prime }}\in
E_{i_{1}i_{2}...i_{s}}=E_{M}\iff \overline{a^{\prime }}=\overline{a}%
\overline{\varepsilon }$ and $\overline{a^{\prime }}\in \overline{a}E_{M}=%
\overline{a}L_{\widehat{M}}$, meaning that $\overline{a}E_{M}=\overline{a}L_{%
\widehat{M}}$\ is the variation domain of the conditionally (i.e., lying
inside $U$) variable $\overline{a}\ \ $-$\ $which, while the $\left\{
i_{1},i_{2},...,i_{s}\right\} $-coordinates remain constant and only the
rest defines the variation, is equivalent to that \ $\pi _{1...\widehat{i_{1}%
}...\widehat{i_{2}}...\widehat{i_{s}}...n}\left( \overline{a^{\prime }}%
\right) =$ $\ $\ \ 

=$\pi _{1...\widehat{i_{1}}...\widehat{i_{2}}...\widehat{i_{s}}...n}\left( 
\overline{a}\overline{\varepsilon }\right) =\pi _{1...\widehat{i_{1}}...%
\widehat{i_{2}}...\widehat{i_{s}}...n}\left( \overline{a}\right) \pi _{1...%
\widehat{i_{1}}...\widehat{i_{2}}...\widehat{i_{s}}...n}\left( \overline{%
\varepsilon }\right) $ =

$=\pi _{1...\widehat{i_{1}}...\widehat{i_{2}}...\widehat{i_{s}}...n}\left( 
\overline{a}\right) \overline{\varepsilon }$, i.e., while $\overline{a}$
varies in the prescribed way, the variation domain of the really changing
part, $\pi _{1...\widehat{i_{1}}...\widehat{i_{2}}...\widehat{i_{s}}%
...n}\left( \overline{a}\right) $, is\ 

$\pi _{1...\widehat{i_{1}}...\widehat{i_{2}}...\widehat{i_{s}}...n}\left( 
\overline{a}\right) E_{i_{1}...i_{2}...i_{s}}$ or, by another notation, $\pi
_{\widehat{M}}\left( \overline{a}\right) L_{\widehat{M}}$. \ 

Conversely, by holding the (complementary) set of $\widehat{M}$-coordinates
of $\overline{a}$ fast, the variation domain of $\overline{a}$ becomes,
similarly, $\overline{a}E_{1...\widehat{i_{1}}...\widehat{i_{2}}...\widehat{%
i_{s}}...n}=\overline{a}E_{\widehat{M}}=\overline{a}L_{M}$\ - \ while that
of its really variable part\ \ $\pi _{i_{1}i_{2}...i_{s}}\left( \overline{a}%
\right) $\ \ shall be \ 

$\pi _{i_{1}i_{2}...i_{s}}\left( \overline{a}\right) E_{1...\widehat{i_{1}}%
...\widehat{i_{2}}...\widehat{i_{s}}...n}=\pi _{M}\left( \overline{a}\right)
L_{M}$.

In this way, any element $\overline{a}\in U$\ determines in this way an
assignment of the coset $\pi _{\widehat{M}}\left( \overline{a}\right) L_{%
\widehat{M}}$ from $\pi _{\widehat{M}}\left( U\right) \diagup L_{\widehat{M}%
} $\ to\ the coset $\pi _{M}\left( \overline{a}\right) L_{M}$\ from\ $\pi
_{M}\left( U\right) \diagup L_{M}$ - and vice versa. \ We are going to show
now that these assignments indeed define an isomorphism.

It is at first clear that, while $\overline{a}\ $varies over $U$, the union
of the cosets

$\pi _{i_{1}i_{2}...i_{s}}\left( \overline{a}\right) E_{1...\widehat{i_{1}}%
...\widehat{i_{2}}...\widehat{i_{s}}...n}=\pi _{M}\left( \overline{a}\right)
L_{M}$ gives the whole of $\pi _{M}\left( U\right) $ and, likewise, the
union of the cosets $\pi _{1...\widehat{i_{1}}...\widehat{i_{2}}...\widehat{%
i_{s}}...n}\left( \overline{a}\right) E_{i_{1}i_{2}...i_{s}}=\pi _{\widehat{M%
}}\left( \overline{a}\right) L_{\widehat{M}}$ gives all of $\pi _{\widehat{M}%
}\left( U\right) $.

\textbf{b.} \ Let us now deal with the bijectivity.

To that end we are also here introducing a special (w.r.t. $M$) relation on $%
U\leq \pi _{M}\left( U\right) \times $ $\pi _{\widehat{M}}\left( U\right) $,
where in this last product we shall denote the typical element $\overline{x}$
as $\overline{x}=\left( \overline{x}_{M};\overline{x}_{\widehat{M}}\right) $
with $\overline{x}_{M}=\pi _{M}\left( \overline{x}\right) $, $\overline{x}_{%
\widehat{M}}=\pi _{\widehat{M}}\left( \overline{x}\right) $; we shall still
allow ourselves to view $\overline{x}_{M}$,$\overline{x}_{\widehat{M}}$ as
elements of the original direct product as well, without notification.
Notice that as such they (as well as their resp. inverses) commute, because
of the direct product.

Let us then say for two such elements $\overline{x}$, $\overline{y}=\left( 
\overline{y}_{M};\overline{y}_{\widehat{M}}\right) $ to be "related" if
either $\overline{x}_{M}=\overline{y}_{M}$\ or $\overline{x}_{\widehat{M}}=%
\overline{y}_{\widehat{M}}$. This rule of course corresponds to our two
"variations" above, while keeping some coordinates fixed, respectively the
one way or the other, and amounts to either $\overline{x}^{-1}\overline{y}%
\in L_{M}$\ or $\overline{x}^{-1}\overline{y}\in L_{\widehat{M}}$. We then
let "$\mathfrak{r}$" be the least transitive relation generated by that
rule, which we may immediately see to be an equivalence relation, in fact
one quite reminiscent of that defined in our proof of theorem 1.

It is then immediate to see that for any given $\overline{x}=\left( 
\overline{x}_{M};\overline{x}_{\widehat{M}}\right) \in U$,\ the $\mathfrak{r}
$-equivalence class $\left[ \overline{x}\right] $ to which $\overline{x}$\
belongs is $\left[ \overline{x}\right] =\left\{ \left( \overline{x}_{M}L_{M};%
\overline{x}_{\widehat{M}}L_{\widehat{M}}\right) \right\} =$

=$\left\{ \left( \overline{x}_{M}l;\overline{x}_{\widehat{M}}l^{\prime
}\right) /l\in L_{M},l^{\prime }\in L_{\widehat{M}}\right\} \subseteq \left(
\pi _{M}\left( U\right) \diagup L_{M}\right) \times \left( \pi _{\widehat{M}%
}\left( U\right) \diagup L_{\widehat{M}}\right) $. By comparing this to the
assignment determined by any element, call it now $\overline{x}$, the way we
did in part \textbf{A} of the proof, and by virtue of the guaranteed
distinctiveness of the equivalent classes, it becomes clear that the
assignments in \textbf{A} altogether amount to a bijective function $\sigma
:\pi _{M}\left( U\right) \diagup L_{M}\longrightarrow $\ \ $\pi _{\widehat{M}%
}\left( U\right) \diagup L_{\widehat{M}}$, with the "structural" property of
assigning $\overline{x}_{M}L_{M}\longmapsto \overline{x}_{\widehat{M}}L_{%
\widehat{M}}$, for any $\overline{x}=\left( \overline{x}_{M};\overline{x}_{%
\widehat{M}}\right) \in U$. This assignment through $\overline{x}=\left( 
\overline{x}_{M};\overline{x}_{\widehat{M}}\right) \in U$ is thus also bound
to $\left( \overline{x}_{M}L_{M},\overline{x}_{\widehat{M}}L_{\widehat{M}%
}\right) =\left( \overline{x}_{M};\overline{x}_{\widehat{M}}\right) L_{M}L_{%
\widehat{M}}=\overline{x}L_{M}L_{\widehat{M}}=\overline{x}\left( L_{M}\times
L_{\widehat{M}}\right) $, hence bijection $\sigma $\ is extended, $\sigma
:\pi _{M}\left( U\right) \diagup L_{M}\longrightarrow $\ \ $\pi _{\widehat{M}%
}\left( U\right) \diagup L_{\widehat{M}}\longrightarrow U\diagup L_{M}\times
L_{\widehat{M}}$ by the totality of the (elementwise "structural")
assignments $\overline{x}_{M}L_{M}\longmapsto \overline{x}_{\widehat{M}}L_{%
\widehat{M}}\longmapsto \left( \overline{x}_{M}L_{M},\overline{x}_{\widehat{M%
}}L_{\widehat{M}}\right) $ for all $\overline{x}\in U$.

\textbf{c.} It remains now only to prove the homomorphic property of the
(extended) $\sigma $.

However this follows from its "structural property" and the componentwise
multiplication in the direct product, i.e. the fact that $\left( \overline{x}%
\overline{y}\right) _{M}=$ $\overline{x}_{M}\overline{y}_{M}$, $\left( 
\overline{x}\overline{y}\right) _{\widehat{M}}=$ $\overline{x}_{\widehat{M}}%
\overline{y}_{\widehat{M}}$\ - and so on. The situation is very similar to
that in our proof of theorem 1 (case $n=2$).

As for the converse, we construct the $\left( M;\widehat{M}\right) $-type
"pair fibres" of the suitable $U$\ thanks to the "structural property" of $%
\sigma $, so that $\pi _{M}\left( U\right) =W_{M}$ and $\pi _{\widehat{M}%
}\left( U\right) =W_{\widehat{M}}$. \ \ \ \ \ \ 
\end{proof}

\bigskip

\begin{theorem}
\label{ssgTh}\bigskip For any non-empty proper subset $M=\left\{
i_{1},i_{2},...,i_{s}\right\} $\ of $\left\{ 1,...,n\right\} $, there is a
unique isomorphism \ $\sigma :$ $\pi _{M}\left( U\right) \diagup L_{M}%
\widetilde{\rightarrow }$\ \ $\pi _{\widehat{M}}\left( U\right) \diagup L_{%
\widehat{M}}$ ($\widetilde{\rightarrow }$\ $U\diagup L_{M}\times L_{\widehat{%
M}}$), which further extends to an isomorphism to $R=U\diagup I$\ if and
only if $I$\ splits over $L_{M}$; if that is not the case, then $R$\ is
isomorphic to a quotient of $U\diagup L_{M}\times L_{\widehat{M}}$.

Furthermore the above isomorphism $\sigma $\ has the "structural" property
that, for any $\overline{a}=\left( a_{1},...,a_{n}\right) \in U,$ \ $\sigma $
sends the coset $\pi _{M}\left( \overline{a}\right) L_{M}$\ over to $\pi _{%
\widehat{M}}\left( \overline{a}\right) L_{\widehat{M}}$, which implies the
partition of $U$ into distinct "pair fibres".

That implies that also the converse of the statement is true, meaning that
we may restore $U$ from the following data: A partition $\left\{
1,...,n\right\} =M\cup \widehat{M}$, subgroups $W_{M}\leq $\ $%
Dr\tprod\limits_{i\in M}A_{i}$\ and $W_{\widehat{M}}\leq $\ $%
Dr\tprod\limits_{i\in \widehat{M}}A_{i}$, a normal subgroup $I$\ of $%
W_{M}\times W_{\widehat{M}}$, which is \textbf{"adhesive"} as a subset of
the original product, i.e. it only contains one "$\sim $"-equivalence
class,\ together with a "structural isomorphism" \ $\sigma :$ $W_{M}\diagup
L_{M}\widetilde{\rightarrow }$\ \ $W_{\widehat{M}}\diagup L_{\widehat{M}}$,
where $L_{M}=\pi _{M}\left( I\right) $, $L_{\widehat{M}}=\pi _{\widehat{M}%
}\left( I\right) $.
\end{theorem}

\begin{proof}
Directly from lemmata \ref{LssgTh}\ \& \ref{LsplC}\ .\ \ \ \ \ \ \ \ \ \ \ \
\ \ \ \ \ \ \ \ \ \ \ \ \ \ \ \ \ \ \ \ \ \ \ \ \ \ \ \ \ \ 
\end{proof}

\bigskip

\ \ \ \ \ \ \ \ \ \ \ \ \ \ \ \ \ \ \ \ \ \ \ \ \ \ \ \ \ \ \ \ \ \ \ \ \ \
\ \ \ \ \ \ \ \ \ \ \ \ \ \ \ \ \ \ \ \ \ \ \ \ \ \ \ \ \ \ \ \ \ \ \ \ \ \
\ \ \ \ \ \ \ \ \ \ \ \ \ \ \ \ \ \ \ \ \ \ \ \ \ \ \ \ \ \ \ \ \ \ \ \ \ \
\ \ \ \ \ \ \ \ \ \ \ \ \ \ \ \ \ \ \ \ \ \ \ \ \ \ \ \ \ \ \ \ \ \ \ \ \ \
\ \ \ \ \ \ \ \ \ \ \ \ \ \ \ \ \ \ \ \ \ \ \ \ \ \ \ \ \ \ \ \ \ \ \ \ \ \
\ \ \ \ \ \ \ \ \ \ \ \ \ \ \ \ \ \ \ \ \ \ \ \ \ \ \ \ 

\begin{remark}
\label{diagR}This theorem may be considered as a generalization of the case
n=2.\bigskip \bigskip

In both cases we may depict the subgroup $U$ diagrammatically as $%
\begin{array}{c}
\diagup \diagdown%
\end{array}%
$ - with $2$ edges. The converse statement in both theorems shows that the
subgroup $U$\ is fully determined once the following data is given: (i) The
2 edge-groups, call them $W_{M}$, $W_{\widehat{M}}$, amounting to the
subgroups $\pi _{M}\left( U\right) $, $\pi _{\widehat{M}}\left( U\right) $\
in the last theorem, (ii) A normal subgroup $I$\ of $W_{M}\times W_{\widehat{%
M}}$, such that, by defining $L_{M}=\pi _{M}\left( I\right) $, $L_{\widehat{M%
}}=\pi _{\widehat{M}}\left( I\right) $ (which shall correspond to the bottom
vertices, call them \textbf{socles}, of the two edges), the factor groups $%
W_{M}\diagup L_{M}$,\ $W_{\widehat{M}}\diagup L_{\widehat{M}}$\ are
isomorphic, both corresponding to the \textbf{"head"} vertex of the two
edges, (iii) An actual such isomorphism \ $\sigma :$ $W_{M}\diagup L_{M}%
\widetilde{\rightarrow }$\ \ $W_{\widehat{M}}\diagup L_{\widehat{M}}$, which
shall serve as the "structural isomorphism" for determining the "pair
fibres" of $U$. The condition for $U$ to be a subgroup in the original
direct product $A_{1}\times A_{2}\times ...\times A_{n}$\ is that $I$\ shall
be \textbf{"cohesive"} as a subset of that, where of course $W_{M}\leq $\ $%
Dr\tprod\limits_{i\in M}A_{i}$\ and $W_{\widehat{M}}\leq $\ $%
Dr\tprod\limits_{i\in \widehat{M}}A_{i}$. Notice that, in case $I$ splits
over $L_{M}$, the cohesive property amounts to just the cohesiveness of $%
L_{M}$ and $L_{\widehat{M}}$.

This picture does actually suggest that a whole "class" of such subgroups
may be defined by just varying the structural isomorphism $\sigma $: This is
actually the subject of subsection 4.2, where we however only actualize the
case that $I$ splits over $L_{M}$, but in the generalized context of theorem %
\ref{gTh}.
\end{remark}

\bigskip

Similarly to the lemma \ref{L2}, we prove also here the following:

\begin{lemma}
\label{LL2}For any\ non-empty proper subset $M$ of $\left\{ 1,...,n\right\} $%
, $\pi _{M}\left( U\right) \cong U\diagup L_{\widehat{M}}$, $\ \pi _{%
\widehat{M}}\left( U\right) \cong U\diagup L_{M}$.\ \ \ \ \ \ \ \ \ \ \ \ \
\ \ \ \ \ \ \ \ \ \ \ \ \ \ \ \ \ \ \ \ \ \ \ \ \ \ \ \ \ \ \ \ \ \ \ \ \ \
\ \ \ \ \ \ \ \ 
\end{lemma}

\begin{proof}
\bigskip Observe again that $\ker \pi _{M}|_{U}=L_{\widehat{M}}$,\ $\ker \pi
_{\widehat{M}}|_{U}=L_{M}$.\ \ \ \ \ \ \ \ \ \ \ \ \ \ \ \ \ \ \ \ \ \ 
\end{proof}

\begin{remark}
\bigskip \bigskip \bigskip \label{dSum3}In continuation of remarks \ref{dSum}
and \ref{dSum2}, by going through the arguments of our proofs in this
section, it is easy to see that they are generalizable to the case of an
infinite direct sum. Notice that corresponding to theorem \ref{ssgTh}, we
shall then have a partition $J=M\sqcup \widehat{M}$\ \ (where $\widehat{M}%
=J\smallsetminus M$, as usually)\ of\ \ the indexing set $J$.

The same generalizability of our results to the infinite case remains true
throughout the following section, however and for space economy we are not
going to point it out again and again in what follows.
\end{remark}

\section{The general structure}

\subsection{\protect\bigskip The general theorems}

\bigskip Let now $U\leq A_{1}\times A_{2}\times ...\times A_{n}$ ($n>1$) be
a subdirect product - i.e., $\pi _{i}\left( U\right) =A_{i}$ for all $i$ 's.

\begin{lemma}
\label{LgTh} Assume that $U$ above is deltoid, i.e. all $E_{i}$ 's are
trivial; then all $A_{i}$ 's are isomorphic to $U$, and there is a system of
("structural") isomorphisms between any two of them, such that $U$ consist
of $n$-tuples of through those isomorphisms corresponding elements. The
converse holds (trivially) too. In particular, the same holds for any
deltoid subcore $L_{\Lambda }$\ of\ $U$, by considering it as a subdirect
product of $\tprod\limits_{i\in \Lambda }\pi _{i}\left( U\right) $.
\end{lemma}

\begin{proof}
By assuming that we might have two elements $\overline{a},\overline{b}\in U$
with one, say the $i$ 'th, coordinate in common and with at least another
coordinate not in common, that would give the contradiction that \ \ $1\neq 
\overline{b}^{-1}\overline{a}\in E_{i}$ . This\ shows that U entirely
consists of mutually disjoint n-tuples $\overline{a}=\left(
a_{1},...,a_{n}\right) $ ; this, combined with the assumption $\pi
_{i}\left( U\right) =A_{i}$ for all $i$'s, establishes a system of bijective
maps between any two of the direct factors $A_{i}$. That these are group
homomorphisms, simply amounts to the group structure and the coordinatewise
multiplication in U.
\end{proof}

\begin{remark}
As $L_{i}\leq E_{j}$ for any $j\neq i$, the hypothesis of the lemma yields
that also all $L_{i}$ 's are trivial.
\end{remark}

\bigskip

For a non-empty, proper subset $M$\ of $\left\{ 1,...,n\right\} $, such that 
$L_{M}$\ is a maximal deltoid subcore of $U$, let $\kappa $\ be any
arbitrary coordinate contained in $M$, and let $M_{\kappa }$\ denote the set 
$M-\left\{ \kappa \right\} $. Then we shall call $L_{M_{\kappa }}$\ \textbf{%
a\ submaximal deltoid subcore of }$U$\textbf{\ }subject to the (apparently
uniquely determined) maximal $M$.

\begin{proposition}
With the above notation, for a submaximal core $L_{M_{\kappa }}$\ let us
simplify the notation by setting $V=\pi _{\widehat{M_{\kappa }}}\left(
U\right) $; we may consider $V$\ as a subgroup of $\tprod\limits_{i\in 
\widehat{M_{\kappa }}}A_{i}$ via its isomorphism to $A_{\kappa
}<\tprod\limits_{i\in \widehat{M_{\kappa }}}A_{i}$ (see lemma \ref{LgTh}).
Then, with the obvious meaning of notation, we have the following:

(i) $V$\ is a subdirect product of $\tprod\limits_{i\in \widehat{M_{\kappa }}%
}A_{i}$.

(ii) $L_{\kappa }\left( V\right) \cong L_{M}\left( U\right) $($=L_{M}$), $L_{%
\widehat{\kappa }}\left( V\right) =L_{\widehat{M}}\left( U\right) $($=L_{%
\widehat{M}}$).

(iii) $I\left( V\right) =L_{\widehat{M}}\left( V\right) \times L_{\kappa
}\left( V\right) \cong I\left( U\right) =L_{\widehat{M}}\times L_{M}$ , $%
V\diagup I\left( V\right) \cong U\diagup I$ and $V\cong U$.

It is thus possible, by continuing just as with the substitution of $V$\ for 
$U$\ here, to substitute a subdirect product $U$\ by another (subdirect in a
subproduct of the original $\tprod\limits_{i=1}^{n}A_{i}$) which is
isomorphic to it and has a similar subdirect structure but with no
non-trivial deltoid subcores.
\end{proposition}

\begin{proof}
(i) is an immediate consequence of the definition of $V$\ and the properties
of projections.

(ii) Notice that $\widehat{\kappa }$\ here means the complement of $\left\{
\kappa \right\} $ in $\widehat{M}\cup \left\{ \kappa \right\} $, i.e. $%
\widehat{M}$.$\ L_{\kappa }\left( V\right) =Ker\pi _{\kappa }|_{\pi _{%
\widehat{M_{\kappa }}}\left( U\right) }$, which is isomorphic to $%
L_{M}\left( U\right) $\ because of lemma \ref{LgTh}. The second one follows
because any element $x$ of $U$,\ with $\pi _{\kappa }\left( x\right) =1$,
belongs to $I\left( U\right) $, hence by lemma \ref{LgTh} also $\pi
_{M}\left( x\right) =1$, therefore $x\in L_{\widehat{M}}\left( U\right) $.

(iii) The equalities follow from\ lemmata \ref{DD} and \ref{LsplC}, then
apply (ii). Now $\pi _{\kappa }\left( V\right) =\pi _{\kappa }\left(
U\right) =A_{\kappa }$ and lemma \ref{LgTh} again shows that $\pi _{\kappa
}\left( U\right) \cong \pi _{M}\left( U\right) $, therefore $\pi _{\kappa
}\left( V\right) \cong \pi _{M}\left( U\right) $, while in this last
isomorphism $L_{\kappa }\left( V\right) $ corresponds to $L_{M}\left(
U\right) $ (as in (ii)), therefore also $\pi _{\kappa }\left( V\right)
\diagup L_{\kappa }\left( V\right) \cong \pi _{M}\left( U\right) \diagup
L_{M}\left( U\right) $\ (1).

On the other hand by theorem \ref{ssgTh} applied twice, $V\diagup I\left(
V\right) \cong $\ $\pi _{\kappa }\left( V\right) \diagup L_{\kappa }\left(
V\right) \cong \pi _{M}\left( U\right) \diagup L_{M}\left( U\right) $ (see
(1)) $\cong U\diagup I$, as $I=I\left( U\right) $ splits over $%
L_{M}=L_{M}\left( U\right) $. But, also theorem \ref{ssgTh}, this last
quotient is also isomorphic to $\pi _{\widehat{M}}\left( U\right) \diagup L_{%
\widehat{M}}\left( U\right) $, while the first in the above sequence of
isomorphisms $V\diagup I\left( V\right) \cong $\ $\pi _{\widehat{M}}\left(
V\right) \diagup L_{\widehat{M}}\left( V\right) $, therefore also $\pi _{%
\widehat{M}}\left( U\right) \diagup L_{\widehat{M}}\left( U\right) \cong $\ $%
\pi _{\widehat{M}}\left( V\right) \diagup L_{\widehat{M}}\left( V\right) $.
If we now observe that $\pi _{\widehat{M}}\left( U\right) =\pi _{\widehat{M}%
}\left( \pi _{\widehat{M_{\kappa }}}\left( U\right) \right) =\pi _{\widehat{M%
}}\left( V\right) $ and $L_{\widehat{M}}\left( U\right) \subseteq L_{%
\widehat{M}}\left( V\right) $, we deduce from this last isomorphism that $L_{%
\widehat{M}}\left( U\right) =L_{\widehat{M}}\left( V\right) $.

Now, the converse in theorem \ref{ssgTh} as explained in remark \ref{diagR}
makes it clear how to define an isomorphism $V\cong U$.
\end{proof}

\bigskip

\begin{condition}
\label{DC}By this proposition we may from now on assume that our subgroup $U$
of\ $Dr\tprod\limits_{i=1}^{n}A_{i}$\ under consideration contains no
non-trivial deltoid subcores.
\end{condition}

\bigskip Let now be given a partition of a subset of the index-set of the
original direct sum, i.e. $\Lambda =M\sqcup N$.

\textbf{As }$L_{\Lambda }=\left( Dr\tprod\limits_{i\in \Lambda }A_{i}\right)
\cap U$\textbf{, our theorem \ref{ssgTh} is applicable to the subgroup }$%
L_{\Lambda }$\textbf{\ of }$Dr\tprod\limits_{i\in \Lambda }A_{i}$\textbf{,
while the definition of }$L_{M}$\textbf{, }$L_{N}$\textbf{\ as subgroups of }%
$L_{\Lambda }\leq Dr\tprod\limits_{i\in \Lambda }A_{i}$\textbf{\ still
remains unchanged inside }$U\leq A_{1}\times A_{2}\times ...\times A_{n}$%
\textbf{\ (since they were already subgroups of }$Dr\tprod\limits_{i\in
\Lambda }A_{i}$\textbf{\ inside }$Dr\tprod\limits_{i=1}^{n}A_{i}$\textbf{);
hence, by that theorem, at any rate is }$L_{M}\times L_{N}\trianglelefteq $%
\textbf{\ }$L_{\Lambda }$\textbf{, while equality here would mean }$\pi
_{M}\left( L_{\Lambda }\right) =L_{M}$\textbf{\ and, equivalently, }$\pi
_{N}\left( L_{\Lambda }\right) =L_{N}$\textbf{, since in this case the
groups of the isomorphism }$\sigma $\textbf{\ in theorem \ref{ssgTh} are
trivial. }In this connection it is important to notice that, considering a
subgroup $U$ of $Dr\tprod\limits_{i=1}^{n}A_{i}$ in the case that $U$ 's
projection on some of the direct factors $A_{i}$ is trivial makes our
analysis too blurry and useless, by short-circuiting it in effect at a
trivial level; consequently one should have to exclude at least those direct
factors $A_{i}$, on which $U$'s projection is trivial, by taking the
subkernel that corresponds to the direct factors, on which the projection of 
$U$\ is non-trivial. Further, analyzing all subkernels of $U$, could take us
closer to a diagrammatic representation of $U$'s structure - which at any
rate is limited by the (complexity of the) structure of the direct factors $%
A_{i}$\ themselves.

-\textit{\ In elementary terms, the condition }$L_{M}\times L_{N}=$\textit{\ 
}$L_{\Lambda }$\textit{\ means (by theorem \ref{ssgTh}, applied on }$%
L_{\Lambda }$\textit{)\ that, for any }$x\in L_{\Lambda }$\textit{, the
element }$\pi _{M}\left( x\right) $\textit{\ of }$Dr\tprod\limits_{i\in
M}A_{i}\ $\textit{also belongs to }$L_{M}$\textit{\ - or, equivalently, }$%
\pi _{N}\left( x\right) \in L_{N}$\textit{. By now viewing }$L_{\Lambda }$ 
\textit{as a subgroup of} $Dr\tprod\limits_{i\in \Lambda }A_{i}$, while
forgetting for a moment about the original $U\leq A_{1}\times A_{2}\times
...\times A_{n}$, we get the following

\begin{lemma}
\label{L12}For $U\leq A_{1}\times A_{2}\times ...\times A_{n}$, $\varnothing
\neq M\subset \left\{ 1,...,n\right\} $, $M\neq \left\{ 1,...,n\right\} $,
the condition$\ \pi _{M}\left( U\right) \leq U$ implies the following (by
theorem \ref{ssgTh} equivalent) facts:

$\pi _{M}\left( U\right) =L_{M}$, $\pi _{\widehat{M}}\left( U\right) =L_{%
\widehat{M}}$ , $U=L_{M}\times L_{\widehat{M}}$
\end{lemma}

\begin{proof}
\bigskip Enhance the preceding discussion with the remark, following from
the definition of the core $I$ of $U$, that $M$ being a proper subset of $%
\left\{ 1,...,n\right\} $ immediately means that the relation $\pi
_{M}\left( U\right) \leq U$\ implies $\pi _{M}\left( U\right) \leq I(U)$,
which then forces\ $\pi _{M}\left( U\right) =L_{M}$.
\end{proof}

- We are now pointing out a relevant implication of part (a) of theorem \ref%
{ssgTh}: in order to prove, on the contrary, that $L_{M}\times L_{N}\neq $ $%
L_{\Lambda }$, it is enough just to find one $x\in L_{\Lambda }$, with the
property that $\pi _{M}\left( x\right) \notin L_{\Lambda }\ \ $(or,
equivalently, $\notin $ $L_{M}$)!

\begin{lemma}
The cohesive components of $I$ intersect each other trivially.
\end{lemma}

\begin{proof}
\bigskip At first, notice that $M\cap N=\varnothing \Rightarrow L_{M}\cap
L_{N}=1$ and, therefore, $L_{M}\cap L_{N}\neq 1\Rightarrow M\cap N\neq
\varnothing $\ . In view of this, combined with the maximality of the
cohesive components from their definition, it will suffice to prove the
following:

"If M, N, P are mutually disjoint non empty index sets, such that $L_{M\cup
P}$,\ $L_{N\cup P}$\ be cohesive, then $L_{M\cup N\cup P}$\ is cohesive too."

\ Assume to the contrary, that there is a non-trivial decomposition $%
L_{M\cup N\cup P}=L_{R}\times L_{S}$ \textbf{(1)} ($R$, $S$ non-empty, $%
R\cap S=\varnothing $, $R\cap S=M\cup N\cup P$.

On account of lemma \ref{L12}, cohesiveness of $L_{M\cup P}$,\ $L_{N\cup P}$
implies that neither $R$ nor $S$ may be contained in either $M\cup P$ or $%
N\cup P$, which in turn means that $R$ as well as $S$ have non-trivial
intersections with $M$ and $N$. So, by means of of lemma \ref{L12}, we get
through (1) a non-trivial decomposition of the subgroups $L_{M\cup P}$,\ $%
L_{N\cup P}$ of $L_{M\cup N\cup P}$, contrary to their cohesiveness.\ \ 
\end{proof}

\begin{proposition}
\label{cohD}\bigskip\ There is always a unique (up to ordering of factors)
decomposition\ $I=L_{N_{1}}\times ...\times L_{N_{m}}$\ of the core I as the
product of its cohesive components; this will\ be referred to as \textbf{the 
}(total)\textbf{\ cohesion decomposition} of the core $I.$
\end{proposition}

\begin{proof}
If $I=L_{1...n}$ is cohesive, then we are done with m=1; otherwise, we
continue examining its factors, until they cannot be any further decomposed,
meaning that they are cohesive. Uniqueness is a consequence of the previous
lemma.
\end{proof}

\ \ \ 

\begin{remark}
Given that one might have the situation $N\subset M$ ($N\neq M$) and still $%
L_{N}=L_{M}$, to a cohesion decomposition is, to begin with, not necessarily
attached a unique partition of $\left\{ 1,...,n\right\} $. To remedy that,
we agree from now on\ (unless otherwise specified) to take the maximal such
subsets of $\left\{ 1,...,n\right\} $.
\end{remark}

\begin{definition}
We shall call a subgroup $U\leq A_{1}\times A_{2}\times ...\times A_{n}$,
for $n>s+2$, an \textbf{r-weakly} \textbf{smashed} one if there exists a
partition of $\left\{ 1,...,n\right\} $\ into subsets of cardinality at
least $r$, such that for every subset $N=\left\{
i_{1},i_{2},...,i_{s}\right\} $ in the partition ($r\leq s$),\ $%
E_{_{i_{1}i_{2}...i_{s}}}$ \ is contained in (the direct product) $%
\tprod\limits_{\kappa \not\notin N}L_{\kappa }$ . (Trivially, for t%
\TEXTsymbol{<}s, \textbf{t-weak}\ smashedness also implies \textbf{s-weak}
smashedness.)
\end{definition}

For $n>2$ the condition that all $E_{ij},$ $i\neq j$,\ be trivial, of course
also implies that all $L_{i}$ 's are trivial.

\begin{lemma}
"1-weakly smashed" means for $U$ the same as "(cohesively) \textbf{smashed".}
\end{lemma}

\begin{proof}
\bigskip "$\Rightarrow $": As the core $I$ is generated by the $E_{i}$ 's, $%
E_{i}$ $\subset $ $L_{1}\times ...\times L_{n}\Rightarrow I\subset
L_{1}\times ...\times L_{n}$, while the converse inclusion is trivial.

"$\Leftarrow $": Trivial.
\end{proof}

\bigskip The following theorem on smashed subdirect products is a crucial
step toward reaching to the theorem about the general case:

\begin{theorem}
\label{smTh}Let $U\leq A_{1}\times A_{2}\times ...\times A_{n}$ \ ($n\rangle
2$)\ \ be a subdirect product (i.e., $\pi _{i}\left( U\right) =A_{i}$ for
all $i$ 's) which is smashed. Then there is a (uniquely determined)
"structural" system of isomorphisms of the $A_{i}\diagup L_{i}$ 's, all
those being isomorphic to $U\diagup I=U\diagup \left( L_{1}\times
L_{2}\times ...\times L_{n}\right) $, \ in a perfect generalization of the
case n=2. This, again, amounts to realizing $U$ as\textbf{\ a fibre product}
of the $A_{i}$ 's over $R:=A_{1}\diagup L_{1}$, with respect to each $A_{i}$
's epimorphism on it, gotten by composing the canonical $A_{i}\rightarrow
A_{i}\diagup L_{i}$, with $A_{i}\diagup L_{i}\rightarrow A_{1}\diagup L_{1}$
from\ the mentioned "structural" system of isomorphisms. The converse is
again true. Also, for any\ subset of indices $i_{1}\langle i_{2}\langle
...\langle i_{s}$ from $\left\{ 1,...,n\right\} $, we have uniquely
determined structural isomorphisms $R\simeq \pi _{i_{1}i_{2}...i_{s}}\left(
U\right) \diagup \left( L_{i_{1}}\times ...\times L_{i_{s}}\right) $.
\end{theorem}

\begin{proof}
Set, now, $A_{i}^{\prime }=A_{i}\diagup L_{i}$ \ and use the previous
proposition for the projection $U^{\prime }$ of $U$, as a subgroup of

$A_{1}^{\prime }\times A_{2}^{\prime }\times ...\times A_{n}^{\prime }$ \ $%
\simeq \left( A_{1}\times A_{2}\times ...\times A_{n}\right) \diagup \left(
L_{1}\times L_{2}\times ...\times L_{n}\right) =$

=$\ \left( A_{1}\times A_{2}\times ...\times A_{n}\right) \diagup I$ ; by
the component-dependent definition of the core it is clear that the core of $%
U^{\prime }$ is trivial, hence also its generating subgroups $E_{i}%
{\acute{}}%
$, therefore we may apply lemma \ref{LgTh} on $U^{\prime }$, then we lift
back to $U$.

As for the converse, we set $I=L_{1}\times L_{2}\times ...\times L_{n}$, $%
A_{i}^{\prime }=A_{i}\diagup L_{i}$, consider $\left( A_{1}\times
A_{2}\times ...\times A_{n}\right) \diagup I$\ $\simeq $\ $A_{1}^{\prime
}\times A_{2}^{\prime }\times ...\times A_{n}^{\prime }$, apply the converse
of the previous proposition\ and lift back.

Alternatively, we could again use the method of determining the "fibres of $%
n $-tuples" (turning out to be cosets of $I$ in $U$) as equivalence classes
in $U$, as we did in the case $n=2$.

As for the last assertion, it suffices to apply theorem \ref{ssgTh}, since $%
E_{1...\widehat{i_{1}}...\widehat{i_{2}}...\widehat{i_{s}}...n}=\pi
_{i_{1}i_{2}...i_{s}}\left( U\right) \cap U$ which, as it lies inside the
"equivalence class" $I\subset U$ (for s\TEXTsymbol{<}n), is the same as $\pi
_{i_{1}i_{2}...i_{s}}\left( U\right) \cap I=\pi _{i_{1}i_{2}...i_{s}}\left(
U\right) \cap \left( L_{1}\times ...\times L_{n}\right) =L_{i_{1}}\times
...\times L_{i_{s}}.$
\end{proof}

\begin{remark}
This theorem may also be viewed as a generalization, in another direction,
of the case $n=2$.
\end{remark}

\begin{corollary}
A smashed subdirect product of $A_{1}\times A_{2}\times ...\times A_{n}$ may
always be taken as a pull-back of n epimorphisms.
\end{corollary}

\begin{corollary}
If $U\leq A_{1}\times A_{2}\times ...\times A_{n}$\ such that,\ for all $%
i\in N=\left\{ i_{1},i_{2},...,i_{s}\right\} $, $E_{i}$ is contained in $%
\left( L_{i_{1}}\times ...\widehat{L}_{i}\times ...\times L_{i_{s}}\right)
\tprod\limits_{\kappa \not\notin N}E_{\kappa }$ (equivalently, just in

$\left( L_{i_{1}}\times ...\times L_{i_{s}}\right) \tprod\limits_{\kappa
\not\notin N}E_{\kappa }$), then the preceding theorem is applicable for

$\pi _{i_{1}...i_{2}...i_{s}}\left( U\right) $ as a subgroup of the direct
product of its projections on $A_{i_{1}},...,A_{i_{s}}$.
\end{corollary}

\bigskip Now we prove the following analogue to lemmata \ref{L2} and \ref%
{LL2}:

\begin{lemma}
\label{LLL2}Let $U\leq A_{1}\times A_{2}\times ...\times A_{n}$ \ ($n\rangle
2$)\ \ be a smashed subdirect product. Then for any\ sequence of indices $%
i_{1}\langle i_{2}\langle ...\langle i_{s}$ from $\left\{ 1,...,n\right\} $,
\ $\pi _{i_{1}i_{2}...i_{s}}\left( U\right) \cong U\diagup
Dr\tprod\limits_{i\notin \left\{ i_{1},i_{2},...,i_{s}\right\} }L_{i}$. In
particular $A_{i}\cong U\diagup L_{1}\times ...\times \widehat{L}_{i}\times
...\times L_{n}$.\ \ 
\end{lemma}

\begin{proof}
Combine lemma \ref{LL2} with smashedness, which implies that ($E_{1...%
\widehat{i_{1}}...\widehat{i_{2}}...\widehat{i_{s}}...n}=$) $%
L_{i_{1}...i_{2}...i_{s}}=L_{i_{1}}\times ...\times L_{i_{s}}$, ($%
E_{i_{1}i_{2}...i_{s}}=$)$L_{1..\widehat{i_{1}}...\widehat{i_{s}}%
...n}=Dr\tprod\limits_{i\notin \left\{ i_{1},i_{2},...,i_{s}\right\} }L_{i}$%
. \ \ \ \ \ \ \ \ \ \ \ \ \ \ \ \ \ \ \ \ \ \ \ \ \ \ \ \ \ \ \ \ \ \ \ 
\end{proof}

\begin{example}
\label{ex24}\bigskip Assume that we have a normal subgroup $B=B_{1}\times
...\times B_{n}$ of a group $G$, $n\geq 2$, hence every single direct factor 
$B_{i}$ is normal in $G$. Let $G\diagup B\simeq R$.

Set $K_{i}=Dr\tprod\limits_{j\neq i}B_{j}$ , \ $\sigma
_{i}:G\twoheadrightarrow G\diagup K_{i}$, $i=1,...,n$, the natural
epimorphisms and, finally, $A_{i}=G\diagup K_{i}$, $A=Dr\tprod%
\limits_{i=1}^{n}A_{i}$.

Then we get a faithful representation $\sigma $ of G as a subdirect product
of the direct product A, as follows:

$\sigma :G\ni g\longmapsto \left( \sigma _{1}\left( g\right) ,...,\sigma
_{n}\left( g\right) \right) \in A$

It is obvious that this is a monomorphism (we are in the just following
showing its injectivity) - and we set $U=\sigma \left( G\right) \leq
A=Dr\tprod\limits_{i=1}^{n}A_{i}$. We will be using the terminology that we
have established above; it is immediate to see the following facts: \ $\pi
_{i}\left( U\right) =\sigma _{i}\left( G\right) =A_{i}$,

$L_{i}=\sigma \left( B_{i}\right) \simeq $(realizable as "=" inside A) $%
\sigma _{i}\left( B_{i}\right) =$(inside $A_{i}$) $B_{i}K_{i}\diagup
K_{i}=B\diagup K_{i}\simeq B_{i},$therefore also,$\sigma ^{-1}\left(
1\right) =\sigma ^{-1}\left( L_{1}\cap L_{2}\right) =\sigma ^{-1}\left(
L_{1}\right) \cap \sigma ^{-1}\left( L_{2}\right) =$ $B_{1}\cap B_{2}=1$\ \
\ \ (alternatively, $\sigma ^{-1}\left( 1\right) =\sigma ^{-1}\left(
\tbigcap\limits_{i=1}^{n}E_{i}\right) =\tbigcap\limits_{i=1}^{n}\sigma
^{-1}\left( E_{i}\right) =\tbigcap\limits_{i=1}^{n}K_{i}=1$ ), proving the
injectivity of \ $\sigma $\textbf{;} on the other hand, $\sigma \left(
B\right) =Dr\tprod\limits_{i=1}^{n}\sigma _{i}\left( B_{i}\right)
=Dr\tprod\limits_{i=1}^{n}L_{i}$ \ and, $\sigma $ being an isomorphism
between G and U, $U\diagup Dr\tprod\limits_{i=1}^{n}L_{i}=\sigma \left(
G\right) \diagup \sigma \left( B\right) \simeq G\diagup B\simeq R$ .\ For
any\ set of indices $i_{1}\langle i_{2}\langle ...\langle i_{s}$ from $%
\left\{ 1,...,n\right\} $, $E_{i_{1}...i_{2}...i_{s}}=\sigma \left(
K_{i_{1}}\cap K_{i_{2}}\cap ...\cap K_{i_{s}}\right) =\sigma \left(
\tprod\limits_{j\notin \left\{ i_{1},...,i_{s}\right\} }B_{j}\right) =$ $%
Dr\tprod\limits_{j\notin \left\{ i_{1},...,i_{s}\right\} }\sigma \left(
B_{j}\right) =Dr\tprod\limits_{j\notin \left\{ i_{1},...,i_{s}\right\}
}\sigma _{j}\left( B_{j}\right) =Dr\tprod\limits_{j\notin \left\{
i_{1},...,i_{s}\right\} }L_{j}$ \textbf{;} in particular,

\ \ $E_{i}=\sigma \left( K_{i}\right) $=$Dr\tprod\limits_{j\neq i}\sigma
\left( B_{j}\right) $=$Dr\tprod\limits_{j\neq i}\sigma _{j}\left(
B_{j}\right) $=$Dr\tprod\limits_{j\neq i}L_{j},$ showing that U is a smashed
subgroup of the direct product A, hence our last theorem 19 applies;
therefore, its core is just being $I=L_{1}\times L_{2}\times ...\times L_{n}$
, we have $U\diagup I$=$U\diagup Dr\tprod\limits_{i=1}^{n}L_{i}\simeq R$ and
then, according to theorem 19, also \ $\pi _{i}\left( U\right) \diagup
L_{i}\simeq R$ - a fact at which we also can arrive directly, as $\pi
_{i}\left( U\right) \diagup L_{i}$=$\left( G\diagup K_{i}\right) \diagup
\left( B_{i}K_{i}\diagup K_{i}\right) \simeq G\diagup B\simeq R$ .

Also from the same theorem, more generally $R\simeq \pi
_{i_{1}i_{2}...i_{s}}\left( U\right) \diagup L_{i_{1}}\times ...\times
L_{i_{s}},$\ for any proper subset $\left\{ i_{1},i_{2},...,i_{s}\right\}
\subset \left\{ 1,...,n\right\} $ $\ \left( i_{1}\langle i_{2}\langle
...\langle i_{s}\right) $; if we assumed that $\pi
_{i_{1}i_{2}...i_{s}}\left( U\right) \leq U$, which would immediately also
imply, due to the core I 's definition, that $\pi
_{i_{1}i_{2}...i_{s}}\left( U\right) =L_{i_{1}i_{2}...i_{s}}\left( U\right) $%
,\ then, according to theorem 5, all three isomorphic factor groups given by
it should be trivial, hence $U=L_{i_{1}i_{2}...i_{s}}\times L_{1..\widehat{%
i_{1}}...\widehat{i_{s}}...n}$, which in our case equals $L_{1}\times
...\times L_{n}$ - which, through the isomorphism $\sigma $, would then
yield that $R\simeq G\diagup B=1$ and $G=B=$ $B_{1}\times ...\times B_{n}$
(compare as well with lemma 12).
\end{example}

\begin{conclusion}
\label{sspr}Given a subgroup $B=B_{1}\times ...\times B_{n}$ of a group $G$, 
$n\geq 2$, so that every single direct factor $B_{i}$ is normal in $G$, we
get a faithful representation $\sigma $ of G as a smashed subdirect product $%
U$ of the direct product $A=Dr\tprod\limits_{i=1}^{n}G\diagup K_{i}$, where $%
K_{i}=Dr\tprod\limits_{j\neq i}B_{j}$.
\end{conclusion}

\begin{example}
As a case of particular interest for our (quite general) example, we may for
example use as $B$ one of the normal subgroups (for example, the maximal of
them, by starting off with all minimal normal subgroups, at least for $G$
finite) of \ a group $G$ given by the following theorem of R. Remak (as
well):
\end{example}

\begin{theorem}
(R. Remak) \cite{RR1} Let $B_{1},...,B_{m}$ (m\TEXTsymbol{>}0) be minimal
normal subgroups of a group G and set $B=\tprod\limits_{i=1}^{m}B_{i}$. Then
there exists a subset $\left\{ i_{1},i_{2},...,i_{n}\right\} \subset \left\{
1,...,m\right\} $, such that $B=B_{i_{1}}\times ...\times B_{i_{n}}$.
\end{theorem}

\begin{problem}
\bigskip \bigskip\ \ \ \ Conclusion \ref{sspr} may also prompt us to the
more general "inverse" problem, of \ investigating the ways to (faithfully)
represent a given group as a subdirect product; of particular interest would
be to get to non-smashed representations. We are looking at this problem in
our last subsection 4.3 here.
\end{problem}

\bigskip Some orientation on this kind of problems in general Group Theory
can be found in \cite{DHP}; it has already been addressed to since 1930 in 
\cite{RR2}.\bigskip \bigskip

\begin{theorem}
\label{gTh}\bigskip Given $U\leq A=A_{1}\times A_{2}\times ...\times A_{n}$,
let $I=L_{N_{1}}\times ...\times L_{N_{m}}$\ be the (total) cohesive
decomposition of its core $I$; denote, also, by $\pi ^{i}$, $i=1,...,m$ , \
the projection from the product $A$ to its subproduct attributed to the
subset $N_{i}$ of $\left\{ 1,...,n\right\} $. Set $R=U\diagup I$; then all
quotients $\pi ^{i}\left( U\right) \diagup L_{N_{i}}$, for $i=1,...,m$ , are
isomorphic\ to each other and to $R$ in a "structural" way, as in our
previous theorems (see theorem \ref{smTh}). $U$ may be realized as a fiber
product (pull-back) of the $\pi ^{i}\left( U\right) $\ 's ("structurally
coordinated") epimorphisms onto $R$; in other words, $U$ may be realized as
a smashed subdirect product.

Also, for any\ sequence of indices $i_{1}\langle i_{2}\langle ...\langle
i_{s}$ from $\left\{ 1,...,m\right\} $ , we have structural isomorphisms $%
R\simeq \pi ^{i_{1}i_{2}...i_{s}}\left( U\right) \diagup \left(
L_{N_{i_{1}}}\times ...\times L_{N_{i_{s}}}\right) $, where \ \ $\pi
^{i_{1}i_{2}...i_{s}}$ denotes the projection from the product $A$ to its
subproduct attributed to the subset $N_{i_{1}}\cup ...\cup N_{i_{s}}$ (a
disjoint union) of $\left\{ 1,...,n\right\} $.

Thanks to the "structural" property of the above isomorphisms to $R$, the
statement is also here true, meaning that we may restore $U$ from the
following data: A partition $\left\{ 1,...,n\right\}
=\tbigcup\limits_{i=1}^{m}N_{i}$, subgroups $W_{i}\leq $\ $%
Dr\tprod\limits_{j\in N_{i}}A_{j}$, a normal adhesive (inside $%
Dr\tprod\limits_{j\in N_{i}}A_{j}$) subgroup $L_{N_{i}}$ of $W_{i}$, such
that all of \ $W_{i}\diagup L_{N_{i}}$\ be isomorphic to $R:=W_{m}\diagup
L_{N_{m}}$,\ together with a sequence $\left( \sigma _{1},...,\sigma
_{m-1}\right) $ of "structural isomorphisms" $\sigma _{i}:W_{i}\diagup
L_{N_{i}}\longrightarrow R$, which shall determine the "adhesive fibres" of
a subgroup $U$,$\ $having core $\ I=L_{N_{1}}\times ...\times L_{N_{m}}\ $%
and such that $\pi ^{i}\left( U\right) =W_{i}$.
\end{theorem}

\begin{proof}
\bigskip Thanks to the commutativity amongst the factors $%
A_{1},A_{2},...,A_{n}$ of the direct product A, we may rearrange them in an
order that fits into the sequence $N_{1},...,N_{m}$ of our partition of $%
\left\{ 1,...,n\right\} $ and renumber; then, by setting $B_{\kappa
}=Dr\tprod\limits_{j\in N_{\kappa }}A_{j}$, $\kappa =1,...,m$, we get $%
A=B_{1}\times B_{2}\times ...\times B_{m}$, indeed a smashed product,
whereupon we now may apply our theorem 11 and get exactly what we are
looking for.
\end{proof}

\bigskip This last theorem \ref{gTh} is a generalization of theorem \ref%
{smTh}; this may also be fruitfully combined with theorem \ref{ssgTh}.

Notice that the cohesive decomposition of the core involved in the theorem
may also be chosen to be an arbitrary one, i.e. not necessarily the total
but a coarser one.

Now we may also give the full generalization of lemmata \ref{L2}, \ref{LL2}
and \ref{LLL2}, which again follows from the totally smashed case of this
last lemma \ref{LLL2}:

\begin{lemma}
\label{LLLL2}\bigskip Same situation as in theorem \ref{gTh}; then for any\
sequence of indices $i_{1}\langle i_{2}\langle ...\langle i_{s}$ from $%
\left\{ 1,...,m\right\} $, \ $\pi ^{^{i_{1}i_{2}...i_{s}}}\left( U\right)
\cong U\diagup Dr\tprod\limits_{i\notin \left\{
i_{1},i_{2},...,i_{s}\right\} }L_{N_{i}}$. In particular $\pi ^{i}\left(
U\right) \cong U\diagup L_{1}\times ...\times \widehat{L}_{i}\times
...\times L_{n}$.\ 
\end{lemma}

\begin{remark}
\bigskip \label{gdiaG}Theorem \ref{gTh} allows us to make diagrammatic
depictions here, analogous to that of remark \ref{diagR}. In this case\ we
may depict the subgroup $U$ diagrammatically as $%
\begin{array}{c}
\diagup _{...}\diagdown%
\end{array}%
$ - with $m$ edges $W_{i}$, $i=1,...,m$, with "socles" $L_{N_{i}}$\ and
heads isomorphic to $R$, where the "structural isomorphisms" of the heads to 
$R$\ shall be needed to restore the "adhesive fibres" in \ $%
Dr\tprod\limits_{i=1}^{m}W_{i}$, which define $U$\ set-theoretically. Its
group structure is then obtained from that of $Dr\tprod%
\limits_{i=1}^{m}W_{i} $.
\end{remark}

\subsection{\protect\bigskip Subdirect $AutR$-classes}

\bigskip In what follows in this subsection we start off at our last theorem %
\ref{gTh}; however we could equally well apply this theory in the situation
theorem \ref{ssgTh}, as already mentioned in our remark \ref{gdiaG}; on the
other hand theorem 1 is a special case of the a cohesive decomposition as in
theorem \ref{gTh}.

Let us so just adopt the notation of theorem \ref{gTh}.

We shall denote by $\mathfrak{P}^{m-1}\left( AutR\right) $ the set of
equivalence classes in $\left( AutR\right) ^{m}$\ under the following
relation: $\left( \sigma _{1},...,\sigma _{m}\right) \thicksim \left( \tau
_{1},...,\tau _{m}\right) $\ iff there exists $\rho \in AutR:\left( \sigma
_{1}\rho ,...,\sigma _{m}\rho \right) $=$\left( \tau _{1},...,\tau
_{m}\right) $. We may also define multiplication in $\left( AutR\right) ^{m}$
in the obvious way, and it is then equally apparent that left multiplication
by another element preserves equivalence of two elements.

\textbf{We need also to identify all }$m$\textbf{\ factor groups }$\pi
^{i}\left( U\right) \diagup L_{N_{i}}$\textbf{\ with }$R$ through some 
\textbf{fixed isomorphisms:} we allow ourselves further a notational
convention, that we may w.r.t. the\ canonical epimorphisms $\xi _{i}:\pi
^{i}\left( U\right) \twoheadrightarrow \pi ^{i}\left( U\right) \diagup
L_{N_{i}}\cong R$\ identify the preimage $\xi _{1}^{-1}\left( r\right) $\ \
of some $r\in R$\ as $rL_{N_{i}}$, which may not cause any confusion,
inasmuch as these coset representatives "$r$" in different $\pi ^{i}\left(
U\right) $ shall never interact with one another. We shall then remember
that, whenever considering the fixed representatives (transversals) "$r$" in
different $\pi ^{i}\left( U\right) $, they do only make a group when
considered modulo $L_{N_{i}}$. This identification may be viewed as a step
toward the idea of "virtuality"/"virtual category".

It is immediate to check through the universal property of a fiber product
that the fiber products $\left\{ \left( \xi _{i},i=1,...,m\right) ;R\right\} 
$ \ and $\left\{ \left( \rho \circ \xi _{i},i=1,...,m\right) ;R\right\} $\
for any $\rho \in AutR$\ are identical; on the other hand we may also look
directly into these fiber products as subsets (and subgroups) of $%
\tprod\limits_{i=1}^{m}\pi ^{i}\left( U\right) \leq A=A_{1}\times
A_{2}\times ...\times A_{n}$, by looking at their adhesive fibres:

We can immediately see that composing all the $m$ canonical epimorphisms $%
\xi _{i}:\pi ^{i}\left( U\right) \twoheadrightarrow \pi ^{i}\left( U\right)
\diagup L_{N_{i}}$\ \textit{(and similarly any other set of epimorphisms }$%
\sigma _{i}\circ \xi _{i}:\pi ^{i}\left( U\right) \twoheadrightarrow \pi
^{i}\left( U\right) \diagup L_{N_{i}}$\textit{\ for some isomorphisms }$%
\sigma _{i}$ of $\pi ^{i}\left( U\right) \diagup L_{N_{i}}\cong R$\textit{)}%
\ with a $\rho \in AutR$\ \textit{from the left} and then taking the
corresponding fiber-product for $\left\{ \left( \rho \circ \xi
_{i},i=1,...,m\right) ;R\right\} $ results in precisely the same subgroup,
as the one gotten as the original pull-back $U$ of $\left\{ \left( \xi
_{i},i=1,...,m\right) ;R\right\} $ over $R$, inasmuch as it gives precisely
the same adhesive fibres: The connective bundles that constitute the
original fiber product $U$ are precisely $\left\{ \left( \xi _{1}^{-1}\left(
r\right) ,...,\xi _{m}^{-1}\left( r\right) \right) ,r\in R\right\} $ which
by the notational convention mentioned above may also be described as$\
\left\{ \left( rL_{N_{1}},...,rL_{N_{m}}\right) ,r\in R\right\} $, wherein
we are also recalling our standard notation around $U$ (see theorem \ref{gTh}%
). By letting $\overline{\rho }$=$\left( \rho ,...,\rho \right) \in \left(
AutR\right) ^{m}$\ act on the $m$ canonical epimorphisms $\xi _{i}$:$\pi
^{i}\left( U\right) \twoheadrightarrow \pi ^{i}\left( U\right) \diagup
L_{N_{i}}$\ through composition as above, we now get the new fibre product
as a subgroup of $\tprod\limits_{i=1}^{m}\pi ^{i}\left( U\right) $,
consisting of the adhesive fibres

$\left\{ \left( \left( \rho \circ \xi _{1}\right) ^{-1}\left( r\right)
,...,\left( \rho \circ \xi _{m}\right) ^{-1}\left( r\right) \right) ,r\in
R\right\} $=

=$\left\{ \left( \xi _{1}^{-1}\left( \rho ^{-1}\left( r\right) \right)
,...,\xi _{m}^{-1}\left( \rho ^{-1}\left( r\right) \right) \right) ,r\in
R\right\} $=

=$\left\{ \left( \xi _{1}^{-1}\left( r^{\rho ^{-1}}\right) ,...,\xi
_{m}^{-1}\left( r^{\rho ^{-1}}\right) \right) ,r\in R\right\} $ which by
substituting $r\rightarrow r^{\rho }$ is rewritten as

$\left\{ \left( \xi _{1}^{-1}\left( r\right) ,...,\xi _{m}^{-1}\left(
r\right) \right) ,r\in R\right\} $=$\left\{ \left(
rL_{N_{1}},...,rL_{N_{m}}\right) ,r\in R\right\} $ which is as a set
precisely $U$, and with the same multiplication. $\ \ \ \ \ \ \ \ \ \ \ \ \ $

We may also more generally let any $\overline{\sigma }=\left( \sigma _{i}%
\text{, }i=1,...,m\right) \in $ $\left( AutR\right) ^{m}$\ act on $%
\tprod\limits_{i=1}^{m}\pi ^{i}\left( U\right) \diagup L_{N_{i}}$, thus
yielding a new subdirect product in $\tprod\limits_{i=1}^{m}\pi ^{i}\left(
U\right) $, out of the original one $U$, obtained as the fiber product $U^{%
\overline{\sigma }}$:$\left\{ \left( \sigma _{i}\circ \xi
_{i},i=1,...,m\right) ;R\right\} $ over $R$: The resulting $U^{\overline{%
\sigma }}$\ consists of the adhesive fibres

$\left\{ \left( \left( \sigma _{1}\circ \xi _{1}\right) ^{-1}\left( r\right)
,...,\left( \sigma _{m}\circ \xi _{m}\right) ^{-1}\left( r\right) \right) 
\text{, }r\in R\right\} $=\ 

=$\left\{ \left( \xi _{1}^{-1}\left( r^{\sigma _{1}^{-1}}\right) ,...,\xi
_{m}^{-1}\left( r^{\sigma _{m}^{-1}}\right) \right) \text{, \ }r\in
R\right\} $=$\left\{ \left( r^{\sigma _{1}^{-1}}L_{N_{1}},...,r^{\sigma
_{m}^{-1}}L_{N_{m}}\right) ,r\in R\right\} $; had we acted with its "$%
\thicksim $"-equivalent m-tuple \ $\overline{\tau }$=$\left( \tau
_{1},...,\tau _{m}\right) $=$\left( \sigma _{1}\rho ,...,\sigma _{m}\rho
\right) $ of\ $R$-automorphisms above, we would again get the exactly same
set of adhesive fibres, now described as $\left\{ \left( r^{\tau
_{1}^{-1}}L_{N_{1}},...,r^{\tau _{m}^{-1}}L_{N_{m}}\right) ,r\in R\right\} $=

=$\left\{ \left( \left( r^{\rho ^{-1}}\right) ^{\sigma
_{1}^{-1}}L_{N_{1}},...,\left( r^{\rho ^{-1}}\right) ^{\sigma
_{m}^{-1}}L_{N_{m}}\right) ,r\in R\right\} $ which, by substituting $%
r\rightarrow r^{\rho }$ is rewritten as $\left\{ \left( r^{\sigma
_{1}^{-1}}L_{N_{1}},...,r^{\sigma _{m}^{-1}}L_{N_{m}}\right) ,r\in R\right\} 
$; thus we have obtained an equality between the subgroups $U^{\overline{%
\tau }}$ and $U^{\overline{\sigma }}$. Conversely, if we had $U^{\overline{%
\sigma }}$=$U^{\overline{\tau }}$ for some $\overline{\sigma }=\left( \sigma
_{i}\text{, }i=1,...,m\right) $, $\overline{\tau }$=$\left( \tau
_{1},...,\tau _{m}\right) \in $ $\left( AutR\right) ^{m}$,\ then their
corresponding sets of adhesive fibres must be identical, i.e.

$\left\{ \left( r^{\sigma _{1}^{-1}}L_{N_{1}},...,r^{\sigma
_{m}^{-1}}L_{N_{m}}\right) ,r\in R\right\} $=$\left\{ \left( r^{\tau
_{1}^{-1}}L_{N_{1}},...,r^{\tau _{m}^{-1}}L_{N_{m}}\right) ,r\in R\right\} $%
, from which we conclude the existence of a unique bijection $\rho
:R\rightarrow R$ such that $\left\{ \left( r^{\tau
_{1}^{-1}}L_{N_{1}},...,r^{\tau _{m}^{-1}}L_{N_{m}}\right) ,r\in R\right\} $=%
$\left\{ \left( \left( r^{\rho ^{-1}}\right) ^{\sigma
_{1}^{-1}}L_{N_{1}},...,\left( r^{\rho ^{-1}}\right) ^{\sigma
_{m}^{-1}}L_{N_{m}}\right) ,r\in R\right\} $, where this last expression is
the fiber product of $\left\{ \left( \sigma _{i}\circ \rho \circ \xi
_{i},i=1,...,m\right) ;R\right\} $\ over $R$, while the first is the one of $%
\left\{ \left( \tau _{i}\circ \xi _{i},i=1,...,m\right) ;R\right\} $\ over $%
R $, forcing $\sigma _{i}\circ \rho $=$\tau _{i}$. That $\rho $ is
homomorphic modulo $I$\ follows from the multiplication rules of the
adhesive fibres.

These findings amount to the following

\begin{proposition}
\label{tPr}In the above described way, two $m$-tuples $\overline{\tau }$, $%
\overline{\sigma }\in \left( AutR\right) ^{m}$ give rise to the same
subgroup if and only if the $m$-tuples $\overline{\rho }$, $\overline{\sigma 
}$ of automorphisms of $R$ are equivalent under the introduced "right
projective" equivalence relation "$\thicksim $". The so determined action of
such a $\overline{\sigma }$\ on $\tprod\limits_{i=1}^{m}\pi ^{i}\left(
U\right) \diagup L_{N_{i}}$ may be realized by changing the coordinated
structural isomorphisms of theorem \ref{gTh} all the way through, a change
effectuated by "twisting" every $\pi ^{i}\left( U\right) \diagup L_{N_{i}}$\
by $\sigma _{i}^{-1}$.\ By denoting the $\left( AutR\right) ^{m}$-orbit of $%
U $ as $\mathfrak{P}\left( U\right) $, we do so finally get an induced
faithful and transitive action of $\mathfrak{P}^{m-1}\left( AutR\right) $\
on $\mathfrak{P}\left( U\right) $, where by $\mathfrak{P}^{m-1}\left(
AutR\right) $\ we mean $\left( AutR\right) ^{m}\diagup \thicksim $ .
\end{proposition}

\bigskip We point out that the twistings above establish the new
"alignments", that determine the new adhesive fibres that constitute $U^{%
\overline{\sigma }}$\ out of those of the original $U$.

\begin{proposition}
Let $\overline{\sigma }=\left( \sigma _{1},...,\sigma _{m}\right) \in \left(
AutR\right) $, where $\sigma _{1}=id_{R}$, and let $\Sigma =\left\langle
\sigma _{2},...,\sigma _{m}\right\rangle $; then $U^{\overline{\sigma }}\cap
U$ consists of the $R^{\Sigma }$-adhesive fibres in $U$, where $R^{\Sigma }$%
\ is the subgroup of $R$, consisting of the $\Sigma $-fixed points on it -
that is $U^{\overline{\sigma }}\cap U=\left\{ \left(
rL_{N_{1}},...,rL_{N_{m}}\right) ,\text{ }r\in R^{\Sigma }\right\} $, by
using our notation explained above. In particular, all groups in the $%
\mathfrak{P}^{m-1}\left( AutR\right) $-orbit $\mathfrak{P}\left( U\right) $
of $U$ contain the core $I=L_{N_{1}}\times ...\times L_{N_{m}}$ of the
original $U$.
\end{proposition}

\begin{proof}
\bigskip It follows from our discussion above, by comparing the adhesive
fibres of $U^{\overline{\sigma }}$ and $U$.
\end{proof}

\bigskip It is clear that every equivalence class in $\left( AutR\right)
^{m} $\ contains a representative of the form of $\overline{\sigma }$\ in
the proposition above, i.e. with the first component equal to $id_{R}$:\ We
shall call it its \textbf{(1-)canonical representative;} further we shall
call the subgroup $\Sigma $\ of $AutR$ generated by all components of the
canonical representative its breadth group. We could have defined
corresponding breadth groups by demanding the i-th automorphism to be
trivial instead; it is nonetheless immediate to check that \textit{any of
these choices results in the same breadth group.}

We want next to examine, whether we can determine \textbf{conditions to
ensure the existence of a homomorphism }$\alpha $\textbf{\ of }$U$\textbf{,
coinduced by }$\overline{\sigma }=\left( \sigma _{1},...,\sigma _{m}\right) $%
\textbf{:} i.e. which acts trivially on the core $I$\ of $U$\ and induces $%
\sigma _{i}$\ on $\pi ^{i}\left( U\right) \diagup L_{N_{i}}$: such one would
probably establish a very convenient isomorphism between $U$\ and $U^{%
\overline{\sigma }}$. Due to the original direct product, this issue boils
down to the corresponding question for every $\pi ^{i}\left( U\right) $
(except that we are now looking for automorphisms $\alpha :\pi ^{i}\left(
U\right) \rightarrow \pi ^{i}\left( U\right) $).

This would in general seem too good to be true: In order to come closer to
some sufficient conditions for such a cute set-up, let us further \textbf{%
assume that }$U$\textbf{\ splits over }$I$\textbf{, i.e. that }$U\cong
R\ltimes I$\textbf{, which again is equivalent to \ }$\pi ^{i}\left(
U\right) =R_{i}\ltimes L_{N_{i}}$\textbf{, for }$i=1,...,m$\textbf{, where }$%
R_{i}\cong R$\textbf{.}

Let us therefore define such a map $\alpha $\ on $U$, by determining its $%
N_{i}$-coordinates through $\pi ^{i}\left( \alpha \left( r\right) \right) =$%
\ $\sigma _{i}\left( r\right) $ or, with exponential notation, $r^{\sigma
_{i}}$,\ for $r\in R_{i}$,\ $\pi ^{i}\left( \alpha \left( l\right) \right)
=l $\ \ for\ $l\in L_{N_{i}}$.\ This is clearly a bijective map; we are now
going to find conditions for it to be homomorphic:\ \ 

Let $x$, $x^{\prime }\in U$, with $\pi ^{i}\left( x\right) =r_{i}l_{i}$, $%
\pi ^{i}\left( x^{\prime }\right) =r_{i}^{\prime }l_{i}^{\prime }$, where $%
l_{i}$, $l_{i}^{\prime }\in L_{N_{i}}$, $r_{i}$, $r_{i}^{\prime }\in R_{i}$%
.\ On the one side we have that $\pi ^{i}\left( \alpha \left( xx^{\prime
}\right) \right) =\alpha \left( \pi ^{i}\left( xx^{\prime }\right) \right)
=\alpha \left( r_{i}l_{i}r_{i}^{\prime }l_{i}^{\prime }\right) =\alpha
\left( r_{i}r_{i}^{\prime }l_{i}^{r_{i}^{\prime }}l_{i}^{\prime }\right) =$\ 
$r_{i}^{\sigma _{i}}r_{i}^{\prime \sigma _{i}}l_{i}^{r_{i}^{\prime
}}l_{i}^{\prime }$ \ (1), on the other is $\pi ^{i}\left( \alpha \left(
x\right) \alpha \left( x^{\prime }\right) \right) =$ $\pi ^{i}\left( \alpha
\left( x\right) \right) \pi ^{i}\left( \alpha \left( x^{\prime }\right)
\right) =\alpha \left( \pi ^{i}\left( x\right) \right) \alpha \left( \pi
^{i}\left( x^{\prime }\right) \right) =\alpha \left( r_{i}l_{i}\right)
\alpha \left( r_{i}^{\prime }l_{i}^{\prime }\right) =$\ $r_{i}^{\sigma
_{i}}r_{i}^{\prime \sigma _{i}}l_{i}^{r_{i}^{\prime \sigma
_{i}}}l_{i}^{\prime }$\ \ (2).\ Then for $\alpha $\ to be\ a homomorphism
one should have\ $\alpha \left( xx^{\prime }\right) =\alpha \left( x\right)
\alpha \left( x^{\prime }\right) $, or equivalently that $\pi ^{i}\left(
\alpha \left( xx^{\prime }\right) \right) =\pi ^{i}\left( \alpha \left(
x\right) \right) \pi ^{i}\left( \alpha \left( x^{\prime }\right) \right) $,
which by (1) \& (2) means $l_{i}^{r_{i}^{\prime }}=l_{i}^{r_{i}^{\prime
\sigma _{i}}}$, that is, $r_{i}^{\prime -1}r_{i}^{\prime \sigma _{i}}$
centralizes $l_{i}$\ $\forall l_{i}\in L_{N_{i}}$, $\forall r_{i}^{^{\prime
}}\in R_{i}$\ - i.e. that every element of the form $r_{i}^{-1}r_{i}^{\sigma
_{i}}$\ in each $R_{i}$ centralizes $L_{N_{i}}$, $i=1,...,m$. Assume now
further that $\overline{\sigma }\in \left( AutR\right) ^{m}$\ is the
1-canonical representative of its class\ in $\mathfrak{P}^{m-1}\left(
AutR\right) $.\ Notice that, with such an 1-canonical $\overline{\sigma }\in
\left( AutR\right) ^{m}$, the condition we have found is automatically
trivially satisfied by $R_{1}$.

\begin{proposition}
a. If $U$\ splits over its core $I$, then the necessary and sufficient
condition for the existence of an isomorphism $\alpha $ from $U$ to $U^{%
\overline{\sigma }}$ (with $\overline{\sigma }$ 1-canonical) acting
trivially on the core $I$\ of $U$\ and inducing $\sigma _{i}$\ on $R_{i}$, $%
i=1,...,m$ is that every element of the form $r_{i}^{-1}r_{i}^{\sigma _{i}}$%
\ in each $R_{i}$ centralizes $L_{N_{i}}$, $i=2,...,m$.

b. Assuming further that, for every $i=2,...,m$, $\sigma _{i}$\ is a
fixed-point free automorphism of $R_{i}$\ and that every $R_{i}$ is either
finite or abelian and Artinian (as a $%
\mathbb{Z}
$-module), the condition above is equivalent to the statement that $R_{i}$
is contained in the center $\mathfrak{Z}\left( \pi ^{i}\left( U\right)
\right) $\ of\ $\pi ^{i}\left( U\right) $, hence that $\pi ^{i}\left(
U\right) =L_{N_{i}}\times R_{i}$, $i=2,...,m$.\ 

\ In particular, that latter is the case if already\ $R$\ itself centralizes 
$I$,\ i.e. if $U\cong R\times I$, which is equivalent to $\pi ^{i}\left(
U\right) =L_{N_{i}}\times R_{i}$, for $i=1,...,m$. \ \ \ 
\end{proposition}

\begin{proof}
\bigskip It is now sufficient to prove part (b).

Assuming that the automorphism\ $\sigma _{i}$\ is fixed-point free, the map
(not a homomorphism, in general)\ $\phi _{i}$ sending $x\in R_{i}$\ to $%
x^{-1}x_{{}}^{\sigma _{i}}\in R_{i}$ is injective: for \ $x^{-1}x^{\sigma
_{i}}=y^{-1}y^{\sigma _{i}}\Leftrightarrow yx^{-1}=\left( yx^{-1}\right)
^{\sigma _{i}}$, whence the assumption on $\sigma _{i}$\ gives $y=x$.

If $R_{i}$ is finite, then $\phi _{i}$ is\ clearly bijective.

Let us now suppose that $R_{i}$ is abelian and Artinian (as a $%
\mathbb{Z}
$-module).

Then $\phi _{i}$ is\ suddenly homomorphic, actually a monomorphism. Then the
Artinian property (DCC) forces the monomorphism $\phi _{i}$ to be
surjective, hence an automorphism: Suppose that $\phi _{i}$ is\ not
surjective. So, there exists some $0\neq y\in R_{i}$ that does not belong to 
$\func{Im}\phi _{i}$, therefore $\phi _{i}\left( y\right) \notin \func{Im}%
\phi _{i}^{2}$. On the other hand $\phi _{i}\left( y\right) $, obviously
belonging to $\func{Im}\phi _{i}$, cannot be $0$, due to $\phi _{i}$'s
injectivity; that proves that the obvious inclusion $\func{Im}\phi
_{i}\supset \func{Im}\phi _{i}^{2}$\ is strict; by a similar argument the
strictness of inclusion continues inductively in the infinite tower $\func{Im%
}\phi _{i}\supset \func{Im}\phi _{i}^{2}\supset \func{Im}\phi
_{i}^{3}\supset \func{Im}\phi _{i}^{4}\supset ...$, which contradicts the
Artinian DCC. Therefore is $\phi _{i}$\ an automorphism.

That means that in both cases of (b) every element of $R_{i}$ may be written
in the form $x^{-1}x^{\sigma _{i}}$, for some $x\in $\ \ $R_{i}$; therefore,
the condition for the existence of a $\overline{\sigma }$-coinduced
homomorphism $\alpha $ of $U$ becomes that every element of $R_{i}$
centralize $L_{N_{i}}$, according to (a), i.e. that $R_{i}$ centralize $%
L_{N_{i}}$.
\end{proof}

\bigskip Notice that, in case the condition of (b) on $R_{i}$, $i=1,...,m$,
is satisfied for all but for $i=\kappa $, then one should prefer the $\kappa 
$-canonical representative of the class of $\overline{\sigma }$\ in $%
\mathfrak{P}^{m-1}\left( AutR\right) $,\ afterwards examine if all $\sigma
_{i}$\ on $R_{i}$ for $i\neq \kappa $\ are fixed-point free.

\begin{remark}
Here is a situation, where the above proposition is applicable, possibly
(depending on $\overline{\sigma }$) its part (b) too: Let $U$\ above be of
finite order $\kappa \lambda $, $\left( \kappa ,\lambda \right) =1$, $%
|R|=\kappa $, $|I|=\lambda $,\ with $I$\ abelian; in that case it is known
(for example. \cite[IV 3.13, Remark]{KB}) that the sequence $%
I\rightarrowtail U\twoheadrightarrow R$ splits - and, consequently (or just
by the same arguments, as $|L_{N_{i}}|$ is a divisor of $|I|$), $\pi
^{i}\left( U\right) \ $splits over\ $L_{N_{i}}$ (i.e., $L_{N_{i}}%
\rightarrowtail \pi ^{i}\left( U\right) \twoheadrightarrow R_{i}$ splits).
\end{remark}

\bigskip

\begin{lemma}
Define $\phi _{i}:R_{i}\rightarrow R_{i}$ as the map\ sending $x$ to $%
x^{-1}x^{\sigma _{i}}$; by assuming $R_{i}$ to be abelian, $\phi _{i}$\
becomes a group homomorphism. Then the restriction of $\sigma _{i}$ to the
image $\phi _{i}\left( R_{i}\right) =\func{Im}\phi _{i}$\ is fixed-point
free.\ 
\end{lemma}

\begin{proof}
\bigskip Observe that $\ker \phi _{i}=R_{i}^{\left\langle \sigma
_{i}\right\rangle }$,\ where $R_{i}^{\left\langle \sigma _{i}\right\rangle }$%
\ is\ the subgroup of $\sigma _{i}$-fixed points, hence we get the natural
isomorphism $\func{Im}\phi _{i}\cong R_{i}\diagup R_{i}^{\left\langle \sigma
_{i}\right\rangle }$ so that the $\sigma _{i}$-action on $\func{Im}\phi _{i}$
be equivalent to the one induced on $R_{i}\diagup R_{i}^{\left\langle \sigma
_{i}\right\rangle }$.\ \ 
\end{proof}

\begin{example}
Let us just take a simple example, just to assist visualization:

Consider the epimorphisms $\xi _{1}:%
\mathbb{Z}
_{15}\ni x\longmapsto x\func{mod}3\in 
\mathbb{Z}
_{3}$\ and $\xi _{2}:%
\mathbb{Z}
_{21}\ni y\longmapsto y\func{mod}3\in 
\mathbb{Z}
_{3}$;\ their fiber product over $%
\mathbb{Z}
_{3}$\ is then the subgroup $U=\left\{ \left( x,y\right) \in 
\mathbb{Z}
_{15}\times 
\mathbb{Z}
_{21}\text{: }x\func{mod}3=y\func{mod}3\right\} $, clearly a subdirect
product of $%
\mathbb{Z}
_{15}\times 
\mathbb{Z}
_{21}$.\ Let $Aut%
\mathbb{Z}
_{3}=\left\langle \sigma \right\rangle $, $\sigma ^{2}=1$; let us determine
the subgroup $U^{\overline{\sigma }}$,\ with $\overline{\sigma }=\left(
1,\sigma \right) \sim \left( \sigma ^{-1},1\right) $\ which, according to
proposition \ref{tPr}, is effectuated through "twisting"\ by $\left( \sigma
,1\right) $: That means, $U^{\overline{\sigma }}=$\ $\left\{ \left(
x,y\right) \in 
\mathbb{Z}
_{15}\times 
\mathbb{Z}
_{21}\text{: }\left( x\func{mod}3\right) ^{\sigma }=y\func{mod}3\right\} $.
\ 

Notice that both $U$\ and\ $U^{\overline{\sigma }}$ convey similar
diagrammatic depictions as $%
\begin{array}{c}
_{_{%
\mathbb{Z}
_{5}}}\diagup ^{^{%
\mathbb{Z}
_{3}}}\diagdown _{_{%
\mathbb{Z}
_{7}}}%
\end{array}%
$\ .
\end{example}

\bigskip

\subsection{Subdirect $\mathfrak{E}$-(in)decomposability}

\bigskip In the spirit of remark \ref{dSum3} we shall rather be speaking of
subdirect sums than products; of course in the case of a finite number of
summands (in which we mostly use the term "product" here) meaning of the two
terms is identical.

One might ask about the decomposability of any arbitrary group as a
subdirect product and what does a particular kind of decomposition mean in
terms of the structure of the group. To meet this kind of questions, also by
getting inspiration from our example \ref{ex24} above, we come to the
propositions below.

Before proceeding to see them, we wish to generalize the notion of a
subdirect group product, by allowing its definition up to isomorphism and
without the restriction about the finite number of direct factors:

\begin{definition}
$U$ shall be called \textbf{a (generalized) subdirect sum} of the family $%
\left\{ A_{j\text{ }}\text{, }j\in J\right\} $ if there exists a
monomorphism $\mu :U\rightarrow \tcoprod\limits_{j\in J}A_{j}$, such that
all composites $\pi _{i}\circ \mu $ \ be epimorphisms, where $\pi
_{i}:\tcoprod\limits_{j\in J}A_{j}\rightarrow A_{i}$ are the canonical
projections.
\end{definition}

\begin{proposition}
\label{SPF}$U$ is a subdirect sum of the family $\left\{ A_{j\text{ }}\text{%
, }j\in J\right\} $ if and only if there exists a family $\left\{ E_{j\text{ 
}}\text{, }j\in J\right\} $ of normal subgroups of $U$, so that $U\diagup
E_{j}\cong A_{j}$, and $\tbigcap\limits_{j\in J}E_{j}=1$.
\end{proposition}

\begin{proof}
\bigskip If the family of normal subgroups is given, in order to define $\mu
:U\rightarrow \tcoprod\limits_{j\in J}U\diagup E_{j}$ it suffices to define
all $\pi _{j}\circ \mu :$\ $U\rightarrow U\diagup E_{i}$; we simply define
them as the canonical maps. It is clear that $\ker \mu =$ $%
\tbigcap\limits_{j\in J}E_{j}=1$, threfore $\mu $\ is monomorphic.

Conversely, given a \ subdirect sum $U$\ as\ in the definition, let $E_{i}$\
($i\in J$) be\ defined the way we have done it earlier, i.e. $E_{i}=\ker
\left( \pi _{i}\circ \mu \right) $; but since $\pi _{i}\circ \mu
:U\rightarrow A_{i}$\ has been assumed to be epi-, we get readily $U\diagup
E_{i}\cong A_{i}$, as wished. On the other hand the kernel of the
monomorphism $\mu $,\ being trivial, is also equal to $\tbigcap\limits_{j\in
J}\ker \left( \pi _{j}\circ \mu \right) =\tbigcap\limits_{j\in J}E_{j}$, and
we are done.
\end{proof}

\begin{corollary}
\label{SI}\bigskip A group cannot be (isomorphically and non-trivially)
written as a subdirect sum if and only if the intersection of all its
non-trivial normal subgroups is non-trivial.
\end{corollary}

\bigskip Such groups may be called \textit{subdirectly indecomposable.}

There is much more that can in a similar manner be derived from the last
proposition; to state them in generality, let $\mathfrak{E}$ be a property
referring to factor groups by normal subgroups of a given group $G$. We
might also refer to $\mathfrak{E}$ as a class\ of groups, and then consider
the normal subgroups of $G$, such that the corresponding factor group belong
to the class $\mathfrak{E}$. We shall call the intersection of all such
normal subgroups \textbf{the }$\mathfrak{E}$\textbf{-residue of }$G$\textbf{.%
} The factor group corresponding to the $\mathfrak{E}$-residue the \ shall
then be called \textbf{the }$\mathfrak{E}$\textbf{-residual of }$G$\textbf{.}
It becomes then immediate to see the following

\bigskip

\begin{proposition}
The necessary and sufficient condition for a group $G$\ to be expressible as
a subdirect sum of groups belonging to the class $\mathfrak{E}$ is that the $%
\mathfrak{E}$-residue of $G$ is trivial - while there are at least two
non-trivial, proper normal subgroups of $G$,\ such that the corresponding
factor group belongs to $\mathfrak{E}$.
\end{proposition}

We mention some examples in the next

\begin{corollary}
If $G$\ is a torsion group, $\varpi _{1}$, $\varpi _{2}$,...,$\varpi _{s}$
pairwise distinct sets of primes, then $G$\ can be expressed as a subdirect
product of respectively $\varpi _{1}$, $\varpi _{2}$,...,$\varpi _{s}$
-groups if and only if $\tbigcap\limits_{i=1}^{s}\mathfrak{O}^{\varpi
_{i}}=1 $, where $G\diagup \mathfrak{O}^{\varpi }$\ is the $\varpi $%
-residual of $G$ (for definition see for example, \cite[3.44]{JR}), $%
\mathfrak{O}^{\varpi }$\ the "$\varpi $-residue".
\end{corollary}

\bigskip Notice that, unless $G$\ is a $\varpi _{\kappa }^{\prime }$ -group
for all $\kappa \in \left\{ 1,...,s\right\} $, the condition $%
\tbigcap\limits_{i=1}^{s}\mathfrak{O}^{\varpi _{i}}=1$\ also forces that at
least one of all those $\varpi _{\kappa }^{{}}$ - residues $\mathfrak{O}%
^{\varpi _{\kappa }}$\ is less than $G$. The result may also be extended to
the case of an infinite collection of $\varpi _{\kappa }$'s:

\begin{corollary}
If $G$\ is a torsion group and $\left\{ \varpi _{n}/n\in 
\mathbb{N}
\right\} $ a collection of pairwise distinct sets of primes, then $G$\ can
be expressed as a subdirect sum of $\tcoprod\limits_{n\in 
\mathbb{N}
}G\diagup \mathfrak{O}^{\varpi _{n}}$. (To secure non-triviality at all
places we may demand that $\mathfrak{O}^{\varpi _{n}}<G$ for all $n\in 
\mathbb{N}
$).
\end{corollary}

\bigskip Another interesting special case is given by the complete reducible
residue $CR(G)$\ of a group $G$, defined as the intersection of all normal
subgroups of $G$, such that the corresponding factor groups be simple. Then
simplicity yields the following:

\begin{corollary}
$CR(G)$\ is the unique normal subgroup of $G$, such that the factor group is
completely reducible (i.e. isomorphic to the direct sum of simple groups).
\end{corollary}

We remind that a group $G$\ is called quasisimple if it is perfect (: $\left[
G,G\right] =G$) and $G\diagup Z\left( G\right) $ is simple. The quasisimple
residual $QCR\left( G\right) $\ of a group $G$\ is defined as the
intersection of all kernels of epimorphisms of $G$\ onto quasisimple groups
(see f.ex. \cite[intr.]{DHP}, where also the last corollary is stated as
well).

\begin{corollary}
A group $G$\ is expressible as a subdirect sum of quasisimple groups if and
only if its quasisimple residual $QCR\left( G\right) $ is trivial. In that
case it follows that it actually becomes a direct or central product.
\end{corollary}

\begin{corollary}
The quasisimple residue $QCR\left( G\right) $ of a group $G$\ is its unique
smallest normal subgroup such that the corresponding factor group is
expressible as a subdirect sum of quasisimple groups.
\end{corollary}

\bigskip

\begin{proposition}
\label{SDec}\bigskip Every group $G$ may be written as a subdirect product
of groups that are either simple or subdirectly indecomposable.
\end{proposition}

\begin{proof}
\bigskip Let us assign to each $x\neq 1$\ in $G$\ a normal subgroup $K_{x}$%
,\ maximal among the normal subgroups not containing $x$; obviously $%
\tbigcap\limits_{x\neq 1}K_{x}=1$. By invoking to the elementary fact of the 
$1-1$ correspondence between the lattices of normal subgroups of $G$\
containing $K_{x}$, and of normal subgroups of\ $G/K_{x}$ (see for example. 
\cite[3.29]{JR}), we see immediately that any normal subgroup of $G/K_{x}$\
has to contain the (non-trivial) canonical image of $x$ in $G/K_{x}$,
otherwise the $1-1$ correspondence would yield a normal subgroup of $G$,
properly containing $K_{x}$, contradicting the maximality of $K_{x}$. In
case there is no normal subgroup of $G/K_{x}$ either, containing the
canonical image of $x$ in $G/K_{x}$, this group is obviously simple, while
otherwise\ is $G/K_{x}$ subdirectly indecomposable (corollary \ref{SI}).
Proposition \ref{SPF} now gives the result.
\end{proof}

\bigskip

\begin{remark}
Of course this proposition does not tell us anything about how interesting
the guaranteed subdirect decomposition might be. For example in the case of
a simple $G$ the proof of the proposition actually gives us a subdirect
representation of G as the diagonal in the product $\tprod\limits_{x\in
G}G_{x}$, where each $G_{x}=G$. That is, in the case of a simple $G$\ 
\textit{(and only in this!)}, the procedure in the proof of proposition \ref%
{SDec} yields a "deltoid subdirect product" (compare our definition \ref%
{SignD}), in the sense that its whole core $I$ is trivial. That is of course
uninteresting, it makes therefore sense to substitute $H$ for any diagonal $%
\Delta H$ like in the case of a simple $G$ in the outcome of the procedure,
in the spirit of condition \ref{DC}, still to get a subdirect decomposition
of the form guaranteed by proposition \ref{SDec} (except only for the case
that the given group was simple) but a more interesting one. We notice also
that such a decomposition is not in general uniquely determined, as the
choice of a maximal $K_{x}$ is not so either. \ 
\end{remark}

\textit{I believe that this new view/realization of our subject may provide
a key to progress on other questions, even in better understanding some
already known results and thereby also enhancing their deepening or
implementation; as a possible such reviewing might be thought the issue of
the lattice of such subgroups, also of the normal subgroups; on this subject
it should anyway be expedient to revisit, among many others of course, \cite%
{MDM},\cite{SUZ},\cite{JT}.}

\bigskip

\section{Subdirect presentations \& applications to homomorphisms}

\bigskip It is in itself interesting to look at subdirect products from
another point of view (and compare), but we may furthermore gain important
new insight and basic results about homomorphisms/endomorphisms on the way,
considerably extending classical/elementary ones; results which hitherto
have (amazingly) remained hidden, while we come across them very naturally
and unrestrained by the present approach. Also our approach here, for which
we already have been predisposed by example 24, remains general - but it
would be interesting and very fruitful, I believe, to apply our results and
techniques in more specific contexts or in concrete situations.

\subsection{The case of two factors}

\bigskip

Let $f_{i}:A\twoheadrightarrow G_{i}$, $i=1,2$, be non-trivial group
epimorphisms (we may always get to epimorphisms, by substituting the target
group with the image of the homomorphism). Let also $\Delta :A\ni
a\longmapsto \left( a,a\right) \in A\times A$ be the diagonal monomorphism, $%
F:A\times A\ni \left( a_{1},a_{2}\right) \longmapsto \left( f_{1}\left(
a_{1}\right) ,f_{2}\left( a_{2}\right) \right) \in G_{1}\times G_{2}$ and
define $u:=F\circ \Delta $, $U:=u\left( A\right) \leq G_{1}\times G_{2}$.

We shall subsequently be using all our previous terminology and symbols
about $U\leq G_{1}\times G_{2}$. Apparently, $\ker \left( F\right) =$ $\ker
\left( f_{1}\right) \times \ker \left( f_{2}\right) \leq A\times A$,
therefore is \ $\ker \left( u\right) =$ $\ker \left( f_{1}\right) \cap \ker
\left( f_{2}\right) $ \ (1).\ \ Of course, the assumed surjectivity of the $%
f_{i}$ 's means that the chosen U indeed is a subdirect product. Due to (1),
the condition \hspace{0in} $\tbigcap\limits_{i=1}^{2}\ker \left(
f_{i}\right) =1$ is equivalent to $u$ being injective; in this case, we may
describe U more explicitly (set-theoretically), as $U=\left\{ \left(
f_{1}\left( a\right) ,f_{2}\left( a\right) \right) \text{ }/\text{ }a\in
A\right\} $.

As $L_{1}=E_{2}=U\cap G_{1}=$ $f_{1}\left( \ker \left( f_{2}\right) \right) $
and $L_{2}=E_{1}=U\cap G_{2}=f_{2}\left( \ker \left( f_{1}\right) \right) $,
theorem \ref{Th1} in this case becomes:

\begin{proposition}
\label{2epimTh}$G_{1}\diagup f_{1}\left( \ker \left( f_{2}\right) \right)
\simeq G_{2}\diagup f_{2}\left( \ker \left( f_{1}\right) \right) \simeq $

$\simeq U\diagup f_{1}\left( \ker \left( f_{2}\right) \right) \times
f_{2}\left( \ker \left( f_{1}\right) \right) $ .
\end{proposition}

By taking $G_{2}=A$ and $f_{2}=id_{A}$, we get

\begin{corollary}
\label{fhth}Given an epimorphism $f_{1}:A\twoheadrightarrow G_{1}$ of
groups, we have $G_{1}\simeq A\diagup \ker \left( f_{1}\right) $
\end{corollary}

which, of course, is the elementary first homomorphism theorem: at this
point, it is essential to \textit{notice that our proof of theorem \ref{Th1}%
, on which proposition \ref{2epimTh} depends, does not apply this
homomorphism theorem!}

In this view, the first isomorphism in Prop. \ref{2epimTh} above is seen to
be a generalization of the first homomorphism theorem; as such one, we
reformulate it here:

\begin{corollary}
Given group homomorphisms $f_{i}:A\longrightarrow G_{i}$, $i=1,2$ , we have $%
\ Im\left( f_{1}\right) \diagup f_{1}\left( \ker \left( f_{2}\right) \right)
\simeq \func{Im}\left( f_{2}\right) \diagup f_{2}\left( \ker \left(
f_{1}\right) \right) $ .
\end{corollary}

Of particular interest is to specialize to the case, in which we have
endomorphisms instead of homomorphisms, when we shall drop surjectivity, so
as to have the same target group $A$ in all cases; it is often convenient to
take isomorphic copies $A_{i}$ of $A$, through fixed isomorphisms: we may
subsequently do it at convenience, even without special notification.

\begin{corollary}
Given $\rho _{i}\in End\left( A\right) $, i=1,2 , \ it holds that $Im\left(
\rho _{1}\right) \diagup \rho _{1}\left( \ker \left( \rho _{2}\right)
\right) \simeq $

$\simeq \func{Im}\left( \rho _{2}\right) \diagup \rho _{2}\left( \ker \left(
\rho _{1}\right) \right) \simeq U\diagup \rho _{1}\left( \ker \left( \rho
_{2}\right) \right) \times \rho _{2}\left( \ker \left( \rho _{1}\right)
\right) $, where $U=u(A)$, as above .
\end{corollary}

\bigskip

\subsection{The case of n (\TEXTsymbol{>}2) factors}

\bigskip

The general outset is just a generalization of the case $n=2$; thus, let $%
f_{i}:A\twoheadrightarrow G_{i}$, $i=1,...,n$ , be non-trivial group
epimorphisms (we can always reduce to the case of epimorphisms, by taking
the images as domains of the $f_{i}$ 's). Let also $\Delta :A\ni
a\longmapsto \left( a,...,a\right) \in A^{n}$ be the diagonal monomorphism, $%
F:A^{n}\ni \left( a_{1},...,a_{n}\right) \longmapsto \left( f_{1}\left(
a_{1}\right) ,...,f_{n}\left( a_{n}\right) \right) \in
Dr\tprod\limits_{i=1}^{n}G_{i}$ and define $u:=F\circ \Delta $, $U:=u\left(
A\right) \leq Dr\tprod\limits_{i=1}^{n}G_{i}$. As before, $\ker \left(
F\right) =$ $Dr\tprod\limits_{i=1}^{n}\ker \left( f_{i}\right) \leq A^{n}$,
whence $\ker \left( u\right) =$ $\tbigcap\limits_{i=1}^{n}\ker \left(
f_{i}\right) $, so that the assumption of its triviality, i.e. that $%
\tbigcap\limits_{i=1}^{n}\ker \left( f_{i}\right) =1$, amount to $u$ 's
injectivity, in which case we may describe $U$ set-theoretically as, $%
U=\left\{ \left( f_{1}\left( a\right) ,...,f_{n}\left( a\right) \right) 
\text{ }/\text{ }a\in A\right\} $. \ \ 

Our previous terminology shall apply to our $U$ here too.\ \ 

$u$ may of course be viewed as a representation of the group $A$ as
subdirect product; by taking our outview from a subdirect product, however,
one might be looking for a suitable $u$, i.e. an $A$ with the right
homomorphisms, to get to such a "presentation" of a given $U$:

\begin{definition}
For a subdirect product $U$ of $Dr\tprod\limits_{i=1}^{n}G_{i}$, we shall be
calling a homomorphism $u=F\circ \Delta :A\longrightarrow
Dr\tprod\limits_{i=1}^{n}G_{i}$ as above, such that $u(A)=U$, a
"presentation of $U$ by homomorphisms"; we shall also denote this subdirect
product $U$ by $[A;\left( f_{1},...,f_{n}\right) ]$. It shall be called "%
\textbf{terse",} if it is injective, i.e. if \ $\tbigcap\limits_{i=1}^{n}%
\ker \left( f_{i}\right) =1$.\ 
\end{definition}

Let $i_{1}\langle i_{2}\langle ...\langle i_{s}$ sequence of indices from $%
\left\{ 1,...,n\right\} $; set $\Lambda =\left\{
i_{1},i_{2},...,i_{s}\right\} $ and write $\left\{ 1,...,n\right\} =\Lambda
\cup \widehat{\Lambda }$, a disjoint union.

Set $K_{i}=\ker \left( f_{i}\right) $ and, for any subset $\Lambda $ of
indices as above, $K_{\Lambda }=K_{i_{1}i_{2}...i_{s}}=K_{i_{1}}\cap
K_{i_{2}}\cap ...\cap K_{i_{s}}$. Set, furthermore, $\xi _{\Lambda }=\pi
_{\Lambda }\circ u$. Clearly, $\xi _{\Lambda }\left( A\right) =\pi _{\Lambda
}\left( U\right) $.

We see immediately that

\begin{lemma}
\label{tL1}(a) $\ker \left( \xi _{\Lambda }\right) =K_{\Lambda }$ \ , (b) $%
L_{\Lambda }\left( U\right) =u\left( K_{\widehat{\Lambda }}\right) =E_{%
\widehat{\Lambda }}\left( U\right) $ \ and (c) The core $I$ of $U$ is, $%
I=\left\langle E_{i}\text{ }/\text{ }i=1,...,n\right\rangle =\left\langle
u\left( K_{i}\right) \text{ }/\text{ }i=1,...,n\right\rangle $=

=$u\left( \left\langle K_{i}\text{ }/\text{ }i=1,...,n\right\rangle \right)
=u\left( \tprod\limits_{i=1}^{n}K_{i}\right) $.
\end{lemma}

\begin{lemma}
\bigskip Given a presentation by homomorphisms $u$ of the subdirect product $%
U$ as above, we can always get to a terse presentation of $U$ as a subdirect
product of the same direct product.
\end{lemma}

\begin{proof}
Since $K_{12...n}=$\ $\tbigcap\limits_{i=1}^{n}\ker \left( f_{i}\right) $ is
contained in the kernel of any $f_{i}$, all $f_{i}$'s factor through $%
\overline{A}=A\diagup K_{12...n}$, giving rise to $\overline{f_{i}}:%
\overline{A}\twoheadrightarrow G_{i}$, $i=1,...,n$ and, thus, a terse
presentation.
\end{proof}

\bigskip It is immediate to verify the following remarkable lemma:

\begin{lemma}
Given a "tersely presented by homomorphisms" smashed subdirect product $%
U=[A;\left( f_{1},...,f_{n}\right) ]$, we can readily get a usual definition
of $U$ as a pull-back out of it. Conversely, given a definition of a
subdirect product $U$ as a pull-back, we get to a terse presentation of it
as $U=[U;\left( p_{1},...,p_{n}\right) ]$, where $p_{i}$ \ is the $i$ 'th
projection from the direct product, which may be considered as trivial in
the sense that the homomorphism $u=\left( p_{1},...,p_{n}\right) \circ
\Delta $ is the identity map on $U$.
\end{lemma}

\begin{proof}
\bigskip We restrict ourselves to show it here for $n=2$ (in which case $U$
is always smashed), as the technic is the same for any $n$; the
generalization for an arbitrary $n>2$ is obtained by use of theorem \ref%
{smTh}.

For the first part, we set $G_{1}\diagup f_{1}\left( \ker \left(
f_{2}\right) \right) \simeq G_{2}\diagup f_{2}\left( \ker \left(
f_{1}\right) \right) :\simeq R$; for the pull-back, we set off from the
epimorphisms $\ \tau _{i}:G_{i}\twoheadrightarrow R$, that have to be chosen
so that together they induce the structural

$\sigma :G_{1}\diagup f_{1}\left( \ker \left( f_{2}\right) \right) 
\widetilde{\longrightarrow }G_{2}\diagup f_{2}\left( \ker \left(
f_{1}\right) \right) $ (: the structural correspondence of pair-fibres), for
example by letting

$G_{2}\diagup f_{2}\left( \ker \left( f_{1}\right) \right) :=R$ and,
denoting with $\pi _{i}:G_{i}\twoheadrightarrow G_{i}\diagup f_{i}\left(
\ker \left( f_{j}\right) \right) $, $\left\{ i,j\right\} =\left\{
1,2\right\} $, the canonical epimorphisms, take $\tau _{1}=\sigma \circ \pi
_{1}$ and $\tau _{2}=\pi _{2}$; i.e. \ $%
\begin{array}{ccc}
U & \longrightarrow & G_{1} \\ 
\downarrow &  & \downarrow \tau _{1} \\ 
G_{2} & \underrightarrow{\tau _{2}} & R%
\end{array}%
$ .

For the other direction, let $U$\ be given as the pull-back of the
epimorphisms $\ \tau _{i}:G_{i}\twoheadrightarrow R$, i.e. $U=\left\{ \left(
g_{1},g_{2}\right) \in G_{1}\times G_{2}\text{ / }\tau _{1}\left(
g_{1}\right) =\tau _{2}\left( g_{2}\right) \right\} $; take $[U;\left(
p_{1},p_{2}\right) ]$, set $K_{i}=\ker \left( f_{i}\right) $ (i=1,2) and
observe that $\ker \left( p_{1}\right) =1\times L_{2}$, $\ker \left(
p_{2}\right) =L_{1}\times 1$, $p_{i}\left( \ker \left( p_{j}\right) \right)
=L_{i}$ ($i\neq j$).
\end{proof}

\bigskip So, we see that homomorphic presentation of a subgroup of a direct
product is not necessarily bound to its smashedness, as its definition as a
pull-back does.

We may apply theorem \ref{ssgTh} to our "homomorphically" presented"
subdirect product $U$, even without demanding our $u$ to be injective
("terse"). Taking into account lemma \ref{tL1}(b), we get:

\begin{proposition}
With the notation above, $\xi _{\Lambda }\left( A\right) \diagup u\left( K_{%
\widehat{\Lambda }}\right) \cong \xi _{\widehat{\Lambda }}\left( A\right)
\diagup u\left( K_{\Lambda }\right) \cong u\left( A\right) \diagup \left(
u\left( K_{\widehat{\Lambda }}\right) \times u\left( K_{\Lambda }\right)
\right) $ .
\end{proposition}

\bigskip For n=2, this amounts to proposition \ref{2epimTh}; let us see what
we get for n=3:

\begin{corollary}
Given the non-trivial group epimorphisms $f_{i}:A\twoheadrightarrow G_{i}$, $%
i=1,2,3$, we have the double isomorphism:

$G_{1}\diagup f_{1}\left( K_{23}\right) \cong \xi _{23}\left( A\right)
\diagup \xi _{23}\left( K_{1}\right) \cong u\left( A\right) \diagup \left(
f_{1}\left( K_{23}\right) \times \xi _{23}\left( K_{1}\right) \right) $ -
and two more, symmetrically.
\end{corollary}

We recall now a previous note, from just above lemma \ref{L12}:

\textit{- In elementary terms, this condition means that }$L_{M}\times
L_{N}= $\textit{\ }$L_{\Lambda }$\textit{\ iff, for any }$x\in L_{\Lambda }$%
\textit{, the element }$\pi _{M}\left( x\right) $\textit{\ of }$%
Dr\tprod\limits_{i\in M}A_{i}\ $\textit{also belongs to }$L_{M}$\textit{\ -
or, equivalently, for any }$x\in L_{\Lambda }$,\textit{\ }$\pi _{N}\left(
x\right) \in L_{N}$\textit{.}

By using this criterion, we get readily to the following

\begin{proposition}
\label{40}\bigskip \bigskip \bigskip Assume that $K_{12...n}=1$, i.e. that $%
u $ is injective (: a terse presentation of $U$) and let $\Lambda =\left\{
i_{1},i_{2},...,i_{s}\right\} =M\cup N$ \ \ ($M\cap N=\varnothing $) be a
partition of the subset $\Lambda $ of $\left\{ 1,...,n\right\} $. The
condition $L_{\Lambda }=L_{M}\times L_{N}$ is in our case of such a
"homomorphically" presented" subdirect product $U$ equivalent to $K_{%
\widehat{\Lambda }}=K_{\widehat{M}}\times K_{\widehat{N}}$ \ (an "internal"
direct product), where $\widehat{M}$, $\widehat{N}$ are the complements of
M, N (respectively) inside the index set $\left\{ 1,...,n\right\} $.
\end{proposition}

\begin{proof}
By implementing the above mentioned criterion, if $L_{\Lambda }=L_{M}\times
L_{N}$ then,\ \textit{for any }$x\in K_{\widehat{\Lambda }}$\textit{\ (}$%
\Leftrightarrow u\left( x\right) \in L_{\Lambda }$)\textit{, }$\pi
_{M}\left( u\left( x\right) \right) =\xi _{M}\left( x\right) \in L_{M}$,
meaning that there is some $x_{M}\in A$, with $u\left( x_{M}\right) \in
L_{M} $ ($\Leftrightarrow f_{i}\left( x_{M}\right) =1$ for any $i\in 
\widehat{M}\Leftrightarrow x_{M}\in K_{\widehat{M}}$), such that $\pi
_{M}\left( u\left( x_{M}\right) \right) =\pi _{M}\left( u\left( x\right)
\right) $; correspondingly, we have that $\pi _{N}\left( u\left( x\right)
\right) =\xi _{N}\left( x\right) \in L_{N}$, meaning that there is some $%
x_{N}\in K_{\widehat{N}}$, such that $\pi _{N}\left( u\left( x_{N}\right)
\right) =\pi _{N}\left( u\left( x\right) \right) $.

We claim now that $u\left( x_{M}x_{N}\right) =u\left( x\right) $; since both
parts do apparently belong to $L_{\Lambda }$ and $\Lambda =M\cup N$, it
suffices to prove that both projections $\pi _{M}$ and $\pi _{N}$, when
applied to them, give the same result.

However, $\pi _{M}\left( u\left( x_{M}x_{N}\right) \right) =\pi _{M}\left(
u\left( x_{M}\right) \right) \pi _{M}\left( u\left( x_{N}\right) \right)
=\pi _{M}\left( u\left( x_{M}\right) \right) =\pi _{M}\left( u\left(
x\right) \right) $, as $x_{N}\in K_{\widehat{N}}\subset K_{M}=\ker (\pi
_{M}\circ u)$ and, similarly, $\pi _{N}\left( u\left( x_{M}x_{N}\right)
\right) =\pi _{N}\left( u\left( x\right) \right) $; hence, as noticed, $%
u\left( x_{M}x_{N}\right) =u\left( x\right) $, which, by invoking to the
injectivity of $u$, implies that $x=x_{M}x_{N}$.

We have thus proven that $K_{\widehat{\Lambda }}\subset K_{\widehat{M}}K_{%
\widehat{N}}$; the other inclusion being apparent, this implies $K_{\widehat{%
\Lambda }}=K_{\widehat{M}}K_{\widehat{N}}$. But, as $M\cap N=\varnothing $
implies $\widehat{M}\cup \widehat{N}=\left\{ 1,...,n\right\} $, we get that $%
K_{\widehat{M}}\cap K_{\widehat{N}}=K_{1...n}=1$, as $U$ 's presentation is
"terse", implying that the product of those two normal subgroups of $K_{%
\widehat{\Lambda }}$ is, indeed, direct.

The converse implication becoming now quite apparent, our proposition has
been proven.
\end{proof}

\begin{corollary}
\bigskip For the tersely presented $U$ above, the condition for it to be
smashed (see def. 11) is, that the subgroup of $A$, generated by the normal
subgroups $K_{i}$, $i=1,...,n$, is the (necessarily direct, due to
terseness) product of all $K_{\widehat{i}}$ 's - i.e., $\tprod%
\limits_{i=1}^{n}K_{i}=\tprod\limits_{i=1}^{n}K_{\widehat{i}}$.
\end{corollary}

We may also specialize to endomorphisms instead of the above homomorphisms $%
f_{i}$ - in which case, of course, we shall have to abandon surjectivity for
the defining homomorphisms, in the general case.

\bigskip

\begin{example}
\bigskip Let $U=[A;\left( f_{1},...,f_{n}\right) ]$, $f_{i}:A%
\twoheadrightarrow G_{i}$, $i=1,...,n$ ($n\geq 2$), where $A$ has a normal
subgroup $B=B_{1}\times ...\times B_{n}$ , hence every single direct factor $%
B_{i}$ is normal in $A$. Let $G\diagup B\simeq R$.

Set $K_{i}=Dr\tprod\limits_{j\neq i}B_{j}$, \ $\ker \left( f_{i}\right)
=K_{i}$, $G=Dr\tprod\limits_{i=1}^{n}G_{i}$.

In the notation that we have introduced in this last section, $K_{\widehat{i}%
}=\tbigcap\limits_{j\neq i}K_{j}=B_{i}$ and, by the corollary above, U is
smashed.

Of course, this here is a re-visiting (and notational updating) of example %
\ref{ex24}.
\end{example}

\begin{example}
\bigskip We construct an example of a non-smashed subdirect product:

\ \ Let $U=[A;\left( f_{1},...,f_{5}\right) ]$, $f_{i}:A\twoheadrightarrow
G_{i}$, $i=1,...,5$ , where $A$ has a normal subgroup $B=B_{12}\times
B_{3}\times B_{4}\times B_{5}$ , and let\ $C_{1\text{ }},C_{2}$\ be normal
subgroups of A contained in $B_{12}$, with trivial intersection (or
otherwise the presentation wouldn't be terse) but so that $C_{1}C_{2}\neq
B_{12}$, with $K_{1}=\ker \left( f_{1}\right) =C_{2}B_{3}B_{4}B_{5}$, $%
K_{2}=\ker \left( f_{2}\right) =C_{1}B_{3}B_{4}B_{5}$, $K_{3}=\ker \left(
f_{3}\right) =B_{12}B_{4}B_{5}$,\ $K_{4}=\ker \left( f_{4}\right)
=B_{12}B_{3}B_{5}$,\ $K_{5}=\ker \left( f_{5}\right) =B_{12}B_{3}B_{4}$. The
presentation is terse, as $K_{12345}=1$.\ 

By lemma 35(b), we have: $L_{1}=u\left( K_{2345}\right) =u\left(
C_{1}\right) $, $L_{2}=u\left( K_{1345}\right) =u\left( C_{2}\right) $, $%
L_{12}=$\ $u\left( K_{345}\right) =u\left( B_{12}\right) $;\ by applying
proposition 40, we see that $L_{12}$ is cohesive, which means that $U$ here
is not smashed. The cohesive components are readily seen to be $L_{12}$, $%
L_{3}=u\left( K_{1245}\right) =u\left( B_{3}\right) $, $L_{4}=u\left(
K_{1235}\right) =u\left( B_{4}\right) $, $L_{5}=u\left( K_{1234}\right)
=u\left( B_{5}\right) $. We may now apply theorem 29, to deduce that:

$[A;\left( f_{1},f_{2}\right) ]\diagup u\left( B_{12}\right) $\ $\simeq
G_{3}\diagup f_{3}\left( B_{3}\right) \simeq G_{4}\diagup f_{4}\left(
B_{4}\right) \simeq G_{5}\diagup f_{5}\left( B_{5}\right) \simeq $

$\simeq \lbrack A;\left( f_{3},f_{4}\right) ]\diagup u\left(
B_{3}B_{4}\right) \simeq $\ $[A;\left( f_{1},f_{2},f_{3}\right) ]\diagup
u\left( B_{12}B_{3}\right) $\ $\simeq ...$

\ \ It is not difficult to make a generalization of this example.
\end{example}

\begin{remark}
If we let $f_{i}:A\twoheadrightarrow G_{i}$, $i=1,...,n$,\ be just group\
homomorphisms, not necessarily surjective, we get obvious actions of $%
Aut\left( A\right) $ and $Dr\tprod\limits_{i}Aut\left( G_{i}\right) $,
further also of $End\left( A\right) $ and $Dr\tprod\limits_{i}End\left(
G_{i}\right) $,\ on the set (/semigroup) of subgroups of $%
Dr\tprod\limits_{i}G_{i}$. Of particular interest is the case, when all $%
G_{i}$'s are isomorphic - in particular, to $A$.
\end{remark}

\bigskip

\section{Through a virtual dualization to diagrams}

\bigskip

As we have seen, subgroups of direct products may be viewed as pull-backs,
i.e., fiber products. Their structure conveys "naturally" diagrammatic
depictions of the form $%
\begin{array}{c}
\diagup |\diagdown%
\end{array}%
$ - with $m$ edges in the general case of theorem \ref{gTh}; we remind
especially the important remarks \ref{diagR} and \ref{gdiaG}.

Although we have not given a proper general definition for diagrams, its
suggested use in this case just corresponds to the structure theorem \ref%
{gTh} and does certainly have the special restriction, that it refers to a
particular representation of a group $U$ as a subdirect product. It has,
nevertheless, the basic characteristic that we might expect of any diagram:
Consisting of just two layers (levels), the vertices of the lower correspond
to (a direct product of) subobjects (subgroups), the vertex on top is a
factor group, namely one that is a factor group in many different ways
according to our general theorem \ref{gTh}. Let us, for our ease, allow
ourselves to call the vertex on top \textit{the head of our depiction of }$U$%
, the direct product of the lower level its \textit{socle}; notice that this
refers only to the particular embedding of $U$\ as a subdirect product.

If we now, conversely, use the investigated structure of such subgroups in
order to deduce properties for such a basic diagram\textbf{,} then our
theorem \ref{gTh}, our analysis of the subdirect product structure and, in
particular, lemmata \ref{L2}, \ref{LL2}, \ref{LLL2} and\ \ref{LLLL2}\ make
it clear that:

\begin{lemma}
\label{pbf}\textbf{a.} Any subdiagram of the suggested subdirect product
representation diagram of such a subgroup $U$ comprising a single edge (or
any proper subset of the set of edges) corresponds to a certain factor group
- but never to a subgroup of $U$. Consequently, there is no proper subgroup
of $U$\ that corresponds to any subdiagram containing the top vertex.

\textbf{b.} The diagrammatic properties of any subdiagram as in (a),
comprising any number of edges, corresponding to a factor group of $U$, are
the same as of the whole diagram of $U$ - i.e., property \textbf{(a)}
"repeats itself".
\end{lemma}

\bigskip Notice that whenever we speak of subdiagrams here, we shall mean
that they are connected (unless otherwise stated) and that they include any
edge of the given one if and only if they also include both its ends.

But there is another major feature to justify taking this kind of simple
diagrams as a major cornerstone for a diagrammatic theory:\ \ Namely, what
makes such a diagrammatic depiction especially interesting and worth
studying for us is its \textbf{virtuality}, in the sense that the multiple
direct factors of the "socle" \textit{are also determined set-theoretically}
( of well defined sections of $U$, in some extended set-theoretical sense)
and not just up to isomorphism - meaning that: \textbf{Their vertices
correspond to well-defined subsets of well-defined subsections.} Notice that
this virtuality could not possibly be claimed just by reference to
pull-backs, as these are only defined up to isomorphism; this is why our
first approach has been through "subgroups of direct products".

It would next be natural to think of considering the dual case, i.e. the
virtual counterpart of push-outs.

This, however, becomes cumbersome in the category of groups: In it
coproducts (/push-outs) are namely realizable by free products (/by
amalgamated \textit{free} products). In particular, in the case of
extensions of an arbitrary group $G$ by an abelian group $A$, the push-out
is the extension that is (functorially) induced by $G$-module homomorphisms $%
A\rightarrow A^{\prime }$, while the ones induced by homomorphisms $%
G^{\prime }\rightarrow G$\ are as usually obtained as pull-backs. That
push-out is a quotient not of the direct, but of the semidirect (see \cite[%
IV 3, exercise 1(b); see also ex. 2]{KB}).

On the other hand, by moving into the category $\mathfrak{Ab}$ of abelian
groups, duality works fine: then the push-out of a family of morphisms $%
S\rightarrow A_{i}$, $i\in I$, is realized as a certain factor group of
their coproduct (direct sum); one may compare this with our example \ref%
{ex24}.

\textit{Our main focus with diagrammatic methods shall therefore from now on
however shift from groups to modules - and to representation theory. We
intend to get a new kind of diagrams there, ones having "virtual properties"
in a sense that generalizes the basic "virtuality" described above. The
original motivation toward the main subject of this article has actually
been this: to begin understanding and substantializing "virtuality" in
modules as well as possible. Then I chose to generalize by considering the
more difficult category of groups, instead of those of abelian groups or
modules, while also viewing it as very interesting for its own sake. }

This shift of area (category) shall also allow us to dualize, so as to get
the virtual counterpart of push-outs.

This and much more is done in \cite{StG1} and its natural continuation in 
\cite{StG}.\bigskip \bigskip \bigskip\

\end{document}